\documentclass[12pt]{amsart}
\usepackage{geometry}                
\usepackage{pdftricks}
\begin{psinputs}
\usepackage{pstricks}
\usepackage{multido}
\end{psinputs}
 
\usepackage{amsthm} 
\usepackage{enumerate}
\usepackage{pst-all}

\usepackage{graphicx}
\usepackage{amssymb}
\usepackage{epstopdf}

\newtheorem{theorem}{Theorem}[section]
\newtheorem{lemma}[theorem]{Lemma} 
\newtheorem{remark}[theorem]{Remark} 
\newtheorem{definition}[theorem]{Definition} 
 
\newtheorem{claim}[theorem]{Claim} 

\def\da{\downarrow}
\def\nda{\not\downarrow}
\def\A{\mathcal{A}}

\def\K{\mathcal{K}}
\def\M{\mathbb{M}}
\def\a{\alpha}

\def\o{\omega}

\title{Finding a field in a Zariski-like structure}
\author{Kaisa Kangas}

\begin{document}

\thanks{My research was supported by Finnish National Doctoral Programme in Mathematics and its Applications}


\maketitle

\begin{center}
Department of Mathematics, University of Helsinki \\
P.O. Box 68, 00014, Finland \\ 
kaisa.kangas@helsinki.fi \\
\vspace{5 mm}
\end{center}

\begin{abstract}
We show that if $\M$ is a Zariski-like structure (see \cite{lisuriart}) and the canonical pregeometry obtained from the bounded closure operator (bcl) is non locally modular, then $\M$ interprets either an algebraically closed field or a non-classical group.
\end{abstract}

Mathematics subject classification: 03C50, 03C98 \\
Key words: group configuration, Zariski geometries, AECs

\tableofcontents

\section{Introduction}

E. Hrushovski and B. Zilber introduced Zariski geometries in \cite{HrZi}.
These structures  generalize the Zariski topology of an algebraically closed field.
According to Zilber \cite{Zi}, the primary motivation behind the notion was to identify all classes where Zilber's trichotomy principle holds. 
The principle can be stated as follows:  if $X$ is a strongly minimal set, then one of the following is true: 
\begin{itemize}
\item The canonical pregeometry on $X$ obtained from the algebraic closure operator is trivial;
\item the canonical pregeometry is locally modular;
\item an algebraically closed field can be interpreted in $X$.
\end{itemize}
Zariski geometries provide a context where the trichotomy holds.
Indeed, in \cite{HrZi}, Hrushovski and Zilber showed that in a non locally modular Zariski geometry, there is an algebraically closed field present.

In \cite{lisuriart}, we presented \emph{Zariski-like} quasiminimal pregeometry structures as non-elementary generalizations of Zariski geometries.
Quasiminimal pregeometry structures (in sense of \cite{monet}) provide a non-elementary analogue for strongly minimal structures from the first order context.
There, the canonical pregeometry is obtained from the bounded closure operator, which corresponds to the algebraic closure operator in the first order case.
We will work in the context of quasiminimal classes, i.e. abstract elementary classes (AECs) that arise from a quasiminimal pregeometry structure (see \cite{monet}).
These classes are uncountably categorical and have both AP and JEP and thus also a universal model homogeneous monster model which we will denote $\M$. 
By \cite{lisuriart}, they also have a perfect theory of independence.

In the present work, we prove the counterpart of the theorem from \cite{HrZi} stating the existence of the field in the non locally modular case:  if $\M$ is a Zariski-like structure with non locally modular canonical pregeometry, then $\M$ interprets either an algebraically closed field or a non-classical group (see \cite{HLS}).
This is Theorem \ref{MAIN} in the present paper.
It demonstrates that the concept of a Zariski-like structure captures the idea of a Zariski geometry in the non-elementary context.
It is an open question whether non-classical groups exist, and their existence would be a remarkable result in itself.
By \cite{groupsongeometries}, the first order theory of such a group would be unstable, but not much is known in a more abstract context.
 
The main reason for inventing the concept of Zariski-like quasiminimal structures was to provide a context where we hope to classify non-elementary geometries.
The main result in \cite{HrZi} is that every very ample Zariski geometry can be obtained from the Zariski topology of a smooth curve over an algebraically closed field.
In \cite{coverart}, we proved that the cover of the multiplicative group of an algebraically closed field of characteristic $0$ is Zariski-like.
An analogue for the result from \cite{HrZi} might be something in the lines that Zariski-like structures resemble the cover.
Moreover, Zariski-like structures might be applied to study Zilber's pseudo-exponentiation (see \cite{pseudo}) or some of the quantum algebras discussed by Zilber in \cite{quantum}.
 
To find the field, we use a generalization of E. Hrushovski's group configuration theorem that works in the non-elementary context and, in the $2$-dimensional case, yields an algebraically closed field, assuming that the model does not interpret a non-classical group.
In \cite{lisuriart}, we presented this theorem in the $1$-dimensional case and applied it to show that  in a Zariski-like structure, a group can be found as long as the canonical pregeometry obtained from the bounded closure operator is non-trivial. 
When proving our main theorem, we will use the elements of this group to find the configuration that gives the field.
 
In \cite{lisuriart}, we developed an independence notion that has all the usual properties of non-forking and is applicable in our setting.
In Section 2 of the present work, we will briefly introduce our setup and recall some results concerning the theory of independence presented in \cite{lisuriart}.
In Section 3, we give the generalization of Hrushovski's group configuration theorem in the non-elementary context, including the $2$-dimensional case.
In Section 4, we show that under the assumption that our model does not interpret non-classical groups, a $2$-dimensional group configuration yields an algebraically closed field.
Finding fields in non-elementary contexts has been previously studied by T. Hyttinen, O. Lessman, and S. Shelah in \cite{HLS}, and we apply some of the methods presented there.
Here, also results by T. Hyttinen from \cite{locmod} prove useful.
In Section 5, we apply the $2$-dimensional case of the group configuration theorem to prove our main theorem (Theorem \ref{MAIN}).
 
\section{The setting}

Throughout this paper, we will be working in the context of quasiminimal classes, studied in  \cite{monet} and \cite{kir}.
In \cite{lisuriart}, section 2 (see also \cite{lisuri}, chapter 2), techniques developed for abstract elementary classes (AECs) were used to obtain an independence calculus that has all the usual properties of non-forking and is applicable in this setting.
We now present the setting and the definitions and results needed in the rest of the paper.

\begin{definition}\label{quasiminclass}
Let $M$ be an $L$-structure for a countable language $L$, equipped with a pregeometry $cl$.
We say that $M$ is a \emph{quasiminimal pregeometry structure} if the following hold (tp denotes quantifier-free $L$-type):
\begin{enumerate}
\item (QM1) The pregeometry  is determined by the language.
That is, if $a$ and $a'$ are singletons and
$\textrm{tp}(a, b)=\textrm{tp}(a', b')$, then $a \in \textrm{cl}(b)$ if and only if $a' \in \textrm{cl}(b')$.
\item (QM2) $M$ is infinite-dimensional with respect to cl.
\item (QM3) (Countable closure property) If $A \subseteq M$ is finite, then $\textrm{cl}(A)$ is countable.
\item (QM4) (Uniqueness of the generic type) Suppose that $H,H' \subseteq M$ are countable closed subsets, enumerated so that $\textrm{tp}(H)=\textrm{tp}(H')$.
If $a \in M \setminus H$ and $a' \in M \setminus H'$ are singletons, then $\textrm{tp}(H,a)=\textrm{tp}(H',a')$ (with respect to the same enumerations for $H$ and $H'$).
\item (QM5) ($\aleph_0$-homogeneity over closed sets and the empty set)
Let $H,H' \subseteq M$ be countable closed subsets or empty, enumerated so that $\textrm{tp}(H)=\textrm{tp}(H')$,
and let $b, b'$ be finite tuples from $M$ such that $\textrm{tp}(H, b)=\textrm{tp}(H',b')$, and let $a$ be a singleton such that $a \in \textrm{cl}(H, b)$.
Then there is some singleton $a' \in M$ such that $\textrm{tp}(H, b,a)=\textrm{tp}(H',b',a').$
\end{enumerate}
\end{definition}
 
From a sufficiently large quasiminimal pregeometry structure $\M$, we can construct an AEC $\K(\M)$ (see \cite{monet}, section 2, and the end of section 2 in \cite{lisuriart}) such that  $\M$ is a monster model for the class.
Using techniques applicable for AECs we obtain for $\mathcal{K}(\M)$ an independence calculus that satisfies all the usual properties of non-forking (\cite{lisuriart}, section 2, see also \cite{lisuri}, chapter 2).
Since $\M$ is a quasiminimal pregeometry structure, it then turns out that the independence notion coincides with the one given by the pregeometry. 
In particular, $U$-ranks (see definitions 2.23 and 2.56 in \cite{lisuriart} coincide with pregeometry dimensions.
As usual, we write $A \da_B C$ for ``$A$ is independent from $C$ over $B$".
We write $A \da B$ for $A \da_\emptyset B$.
     
In the AEC setting, automorphisms of the monster model $\M$ play a crucial role.
We will denote the group of all these automorphisms by $Aut(\M)$.
If $A \subset \M$, we will write $Aut(\M/A)$ for the group of automorphisms fixing the set $A$ pointwise.

\begin{definition}
We define \emph{Galois types} as orbits of automorphisms of  $\M$.
For $A \subset \M$, we write $t^g(b/A)=t^g(c/A)$ if the tuples $b$ and $c$ have the same Galois type over $A$, i.e. if there is some automorphism $f \in Aut(\M/A)$ such that $f(b)=c$.
\end{definition}

\begin{remark}\label{le31}
By (e.g.) Lemma 3.1 in \cite{monet}, the quantifier-free first order types imply Galois types over closed sets and finite sets.

Moreover, as seen in \cite{lisuriart} (discussion after Definition 2.80), (QM4) and (QM5) will hold also if the quantifier-free $L$-types are replaced by Galois types (see also \cite{lisuri}, after 2.5.7). 
\end{remark}
 
\begin{definition}
We say that a set $A$ is \emph{bounded} if $\vert A \vert <  \vert \M \vert$.
\end{definition}

\begin{definition}
We say an element $a$ is in the \emph{bounded closure} of $A$, denoted $a \in \textrm{bcl}(A)$, if $t^g(a/A)$ has only boundedly many realizations, i.e. if the set 
$$\{x \in \M \quad \vert \quad t^g(x/A)=t^g(a/A)\}$$ is bounded.

We say that $a$ and $b$ are \emph{interbounded} if $a \in \textrm{bcl}(b)$ and $b \in \textrm{bcl}(a)$.
\end{definition}

\begin{remark}
By Remark 2.88 in \cite{lisuriart},   
if $\M$ is a quasiminimal pregeometry structure, then for any $A \subset \M$, it holds that $bcl(A)=cl(A)$.
In particular, $bcl$ gives a pregeometry on $\M$ and the independence calculus is obtained from it.
Thus, the situation is analogous to the first order strongly minimal case, where the pregeometry 
induced by the algebraic closure operator gives the independence calculus.
\end{remark}

\begin{definition}
We say that a tuple $a$ is \emph{Galois definable} from a set $A$, if it holds for every $f \in \textrm{Aut}(\mathbb{M}/A)$
that $f(a)=a$.
We write $a \in \textrm{dcl}(A)$, and say that $a$ is in the \emph{definable closure} of $A$.

We say that $a$ and $b$ are \emph{interdefinable} if $a \in \textrm{dcl}(b)$ and $b \in \textrm{dcl}(a)$.
We say that they are interdefinable over $A$ if  $a \in \textrm{dcl}(Ab)$ and $b \in \textrm{dcl}(Aa)$.
\end{definition}

\begin{definition}
We say a set $B \subset \M^n$ is \emph{Galois definable} over $A \subset \M$ if for every $f \in Aut(\M/A)$, it holds that $f(B)=B$.
\end{definition} 

It turns out that in our setting, bounded Galois definable sets are countable (by Lemmas 2.24 and 2.26 in \cite{lisuriart}).
 
Our main notion of type will  be that of the weak type rather than the Galois type:

\begin{definition}
Let $A \subset \M$.
We say $b$ and $c$ have the same \emph{weak type} over $A$, denoted $t(b/A)=t(c/A)$, if for every finite $A_0 \subseteq A$, it holds that $t^g(b/A_0)=t^g(c/A_0)$.
\end{definition}

An analogue for the strong types of the first-order context is provided by Lascar types.
 
\begin{definition}
Let $A$ be a finite set, and let $E$ be an equivalence relation on $M^{n}$, for some $n<\o$.
We say $E$ is \emph{$A$-invariant} if for all $f \in Aut(\M /A)$ and $a,b \in \M$, it holds that if $(a,b) \in  E$, then $(f(a),f(b)) \in E$.
We denote the set of all $A$-invariant equivalence 
relations that have only boundedly many equivalence classes by $E(A)$.
 
We say that $a$ and $b$ have the same \emph{Lascar type} over a set $B$, denoted  $Lt(a/B)=Lt(b/B)$, 
if for all finite
$A\subseteq B$ and all $E\in E(A)$, it holds that $(a,b)\in E$.
\end{definition}

\begin{remark}
By Lemma 2.37 in \cite{lisuriart}, Lascar types imply weak types.
 
Moreover, Lascar types are stationary by Lemma 2.46 in \cite{lisuriart}.
\end{remark} 

If $p$ is a stationary type over $A$ and $A \subset C$, we write $p \vert_C$ for the (unique) free extension of $p$ into $C$.
 
 \begin{definition}\label{genericeka}
 Let $B \subset \M$.
 We say an element $b \in B$ is \emph{generic} over some set $A$ if $dim(b/A)$ is maximal (among the elements of $B$). 
 The set $A$ is not mentioned if it is clear from the context.
 For instance, if $B$ is assumed to be Galois definable over some set $D$, then we usually assume $A=D$. 
  
Let $p=t(a/A)$ for some $a \in \M$ and $A \subset B$.
We say $b \in \M$ is a \emph{generic} realization of $p$ (over $B$) if $dim(b/B)$ is maximal among the realizations of $p$.
\end{definition}

\begin{definition}
We say that a sequence $(a_{i})_{i<\a}$ is \emph{indiscernible} over $A$
if every permutation of the sequence $\{ a_{i}\vert\ i<\a\}$
extends to an automorphism $f\in \textrm{Aut}(\M /A)$.
 
We say a sequence $(a_{i})_{i<\a}$ is \emph{strongly indiscernible} over $A$ if for all
cardinals $\kappa$, there are $a_{i}$, $\a\le i<\kappa$, such that
$(a_{i})_{i<\kappa}$ is indiscernible over $A$.
\end{definition}

\begin{lemma}\label{si}
Let $(a_i)_{i<\kappa}$ be a sequence independent over $b$ and suppose $Lt(a_i/b)=Lt(a_j/b)$ for all $i,j<\kappa$.
Then, it is strongly indiscernible over $b$.
\end{lemma} 

\begin{proof}
Let $\A$ be a model such that $b \in \A$ and $(a_i)_{i<\kappa} \da_b \A$.
Then, $Lt(a_i/\A)=Lt(a_j/\A)$, and by Lemma 2.38 in \cite{lisuriart}, the sequence is Morley over $\A$ in the sense of Definition 2.23 in \cite{lisuriart}, and thus strongly indiscernible over $\A$ by Lemma 2.26 in \cite{lisuriart}.
Hence, it is strongly indiscernible over $b$.
\end{proof}

\begin{lemma}\label{morleycofinal}
Let $\kappa$ be an uncountable cardinal such that $cf(\kappa)=\kappa$, and let  $(a_i)_{i<\kappa}$ be a sequence independent over $b$.
Then, there is some $X \subseteq \kappa$, cofinal in $\kappa$, such that $(a_i)_{i \in X}$ is strongly indiscernible over $b$.
\end{lemma}

\begin{proof}
By Lemma 2.27 in \cite{lisuriart}, there is a model $\A$ and a set $X \subseteq \kappa$ cofinal in $\kappa$ such that  $b \in \A$ and $(a_i)_{i \in X}$ is Morley over $\A$ (in the sense of Definition 2.26 in \cite{lisuriart}).  By Lemma 2.29 in \cite{lisuriart}, it is strongly indiscernible over $\A$ and hence over $b$.
\end{proof}

\subsection{$\M^{eq}$  and canonical bases}

In our arguments, we will need the notion of a canonical base.
Thus, we briefly discuss constructing $\M^{eq}$ and finding canonical bases in our setting. 
See \cite{lisuriart}, section 2.4 for details (also, \cite{lisuri}, section 2.4).  

Let $\mathcal{E}$ be a countable collection of $\emptyset$-invariant
equivalence relations $E$ such that $E \subseteq \M^n \times \M^n$ for some $n$.
We assume that the identity relation $=$ is in $\mathcal{E}$ (by Lemma 2.14 in \cite{lisuriart} there are only countably many Galois types over $\emptyset$).  
We let 
$$\M^{eq}= \{ a/E\vert\ a\in\M ,\ E\in\mathcal{E}\},$$ 
and we identify each element $a$ with $a/=$. 
For each $E \in \mathcal{E}$, we add to
our language a predicate $P_{E}$
with the interpretation $\{ a/E \, \vert \,a\in\M\}$
and a function $F_{E}: \M^n \to \M^{eq}$ (for a suitable $n$) such that $F_{E}(a)=a/E$.  
Then, we have all the structure of $\M$ on $P_{=}$.  

In \cite{lisuriart}, section 2.4 (Theorem 2.75 in particular, see also Theorem 2.70), 
it is shown that if all the conditions required to obtain the independence calculus are satisfied on $\M$,
then they are also satisfied on $\M^{eq}$.
By Lemma 2.87 in \cite{lisuriart}, they hold in the case that $\M$ is a quasiminimal pregeometry structure,
and thus we can extend the independence calculus to $\M^{eq}$.
Moreover, the perfect independence calculus can be obtained from the bounded closure operator (bcl) also in the case of $\M^{eq}$ (Theorem 2.70 in \cite{lisuriart}).  
Thus, the setting is analogous to the first order strongly minimal context also in the case of $\M^{eq}$,
and Theorem 2.70 in \cite{lisuriart} guarantees that dimensions can be calculated in a manner analogous to a pregeometry
(Theorem 2.70 (xii) states that if $a \in \textrm{bcl}(B) \setminus \textrm{bcl}(A)$, then $a \nda_A B$,
which implies a condition similar to the exchange property of pregeometries: if $a$ and $b$ are 1-dimensional and
$a \in \textrm{bcl}(Ab)\setminus \textrm{bcl}(A)$, then $b \in \textrm{bcl}(Aa)$).
In what follows, all dimensions will refer to $U$-ranks (see definitions 2.23 and 2.56 in \cite{lisuriart}),
but for the sake of calculations, they can be thought of as dimensions obtained from the bounded closure operator bcl. 

However, in our context, we cannot construct $\M^{eq}$ so that it is both $\o$-stable (in the sense of AECs) and admits elimination of imaginaries. 
Thus, when needed, we just pass into $(\M^{eq})^{eq}, ((\M^{eq})^{eq})^{eq},$ etc.
Lemma 2.87, Theorem 2.75 and Theorem 2.70 in \cite{lisuriart} guarantee that the pregeometry obtained from the bounded closure extends to all of these and 
always gives and independence calculus with all the usual properties of non-forking.
To simplify notation, we will denote all of these just by $\M^{eq}$.
  
Let $\M'$ be a $\vert \M' \vert$ -model homogeneous and universal structure such that $\M$ is a closed submodel of $\M'$ and $\vert \M' \vert > \vert \M \vert$.
We call  $\M'$ the \emph{supermonster}. 
Then, every $f \in \textrm{Aut}(\M)$ extends to some $f' \in \textrm{Aut}(\M')$.
We will usually abuse notation and write just $f$ for both maps. 
 
\begin{definition}
By a \emph{global type} $p$, we mean a maximal collection $\{p_A \, | \, A \subset \M \textrm{ finite }\}$ such that $p_A$ is a Galois type over $A$, and whenever $A \subseteq B$ and $b \in \M$ realizes $p_B$, then $b$ realizes also $p_A$.
We denote the collection of global types by $S(\M)$.
Moreover, we require that global types are consistent, i.e. that for each $p \in S(\M)$, there is some $b \in \M'$ such that $b$ realizes $p_A$ for every finite set $A \subset \M$ (note that the \emph{same} element $b \in \M'$ is required to realize $p_A$ for every $A$).
\end{definition}
  
Let $f \in \textrm{Aut}(\M^{eq})$, $p \in S(\M)$.
We say that $f(p)=p$ if for all finite $A,B \subset \M$ such that $f(B)=A$ and all $b$ realizing $p_B$, it holds that $t(b/A)=p_A$.

\begin{definition}
Let $p \in S(\M)$.
We say that $\alpha \in \M^{eq}$ is a \emph{canonical base} for $p$ if it holds for every $f \in \textrm{Aut}(\M^{eq})$ that $f(p)=p$ if and only if $f(\alpha)=\alpha$.
\end{definition}

In \cite{lisuriart} (Lemma 2.72 and discussion before), it is shown that the collection $\mathcal{E}$ of equivalence relations can be chosen so that every type has a canonical base in $\M^{eq}$.

\begin{definition}
Let $a \in \M$ and let $A \subset \M$.
Let  $b \in \M'$ be such that $Lt(a/A)=Lt(b/A)$ and $b \da_A \M$.
Let $p=t(b/\M)$.
By a \emph{canonical base} for $a$ over $A$, we mean a canonical base of $p$.
We write $\alpha=Cb(a/A)$ to denote that $\alpha$ is a canonical base of $a$ over $A$.
\end{definition}
 
 It can then be shown that in our setting, the canonical bases have the following usual properties.
 \begin{itemize} 
\item
Let $a \in \M$ and let $A \subset \M$ be a finite set.
Then, $Cb(a/A) \in \textrm{bcl}(A)$. (Lemma 2.74 in \cite{lisuriart})
  
\item
Let  $a \in \M$ and let $\alpha=Cb(a/A)$.
Then, $a \da_\alpha A$. (Lemma 2.76 in \cite{lisuriart})
 
\item
Let $\alpha=Cb(a/A)$.
Then, $t(a/\alpha)$ is stationary. (Lemma 2.78 in \cite{lisuriart})
 
\item 
Let $a \in \M$, and let $A$ and $B$ be sets such that $A \subsetneq B$, and let $\alpha \in \M^{eq}$.
If $a \da_A B$, then $\alpha=Cb(a/A)$ if and only if $\alpha=Cb(a/B)$. (Remark 2.75 in \cite{lisuriart})
\end{itemize}
 
\section{The group configuration}
 
To find a field in a Zariski-like structure, we will apply a theorem saying that a group can be interpreted in a model whenever there is a certain configuration of elements present.
The first-order version was part of E. Hrushovski's Ph.D. thesis and can be found in e.g. \cite{pi}, section 5.4.
If the group elements are of dimension $2$, then also a field can be found.
In \cite{lisuriart}, section 3, Hrushovski's group configuration theorem is adapted for quasiminimal classes (see also \cite{lisuri}, Chapter 3). However, only the case resulting in a one-dimensional group (the case $n=1$ in Definition \ref{conf}) is treated.
It generalizes quite easily to the two-dimensional case ($n=2$ in Definition \ref{conf}).
In this section, we will present the proof and refer the reader to \cite{lisuriart} (or \cite{lisuri}) for omitted details. 
 
We now present the configuration that will yield a group.

 \begin{figure}[htbp] \label{fig:multiscale}
\begin{center}
   \includegraphics[scale=0.7]{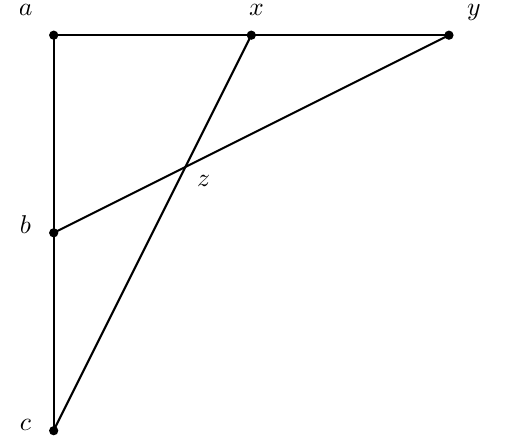}
\end{center}
\end{figure}

\begin{definition}\label{conf}
By a \emph{strict bounded partial quadrangle} over  a finite set $A$ we mean a $6$-tuple of elements $(a,b,c,x,y,z)$ in $\M^{eq}$ such that for some $n \in \{1,2\}$,
\begin{enumerate}[(i)]
\item $dim(a/A)=dim(b/A)=dim(c/A)=n$, and $dim(x/A)=dim(y/A)=dim(z/A)=1$;
\item any triple of non-collinear points is independent over $A$ (see the picture);
\item $dim(a,b,c/A)=2n$;
\item $dim(a, x,y/A)=dim(b,z,y/A)=dim(c,z,x/A)=n+1$; 
\item $b$ is interbounded with $Cb(yz/Ab)$ over $A$.
\end{enumerate}
\end{definition}

\begin{remark}
If each of $a$,$b$,$c$,$x$,$y$,$z$ is replaced by an element interbounded with it over $A$, then  it is easy to see that the new $6$-tuple $(a',b',c',x',y',z')$ is also a strict bounded partial quadrangle over $A$. 
 
We say that this new partial quadrangle is \emph{boundedly equivalent} to the first one.

Also, if $n=1$, then (v) follows from the other conditions in the definition.
\end{remark}

\begin{definition}
We say that a group $G$ is \emph{Galois definable} over $A$ if $G$ and the group operation on $G$ are both Galois definable over $A$ as sets.
\end{definition}

\begin{definition}
Let $S$ be a Galois definable set (over $A$).
We say $S$ has \emph{unique generics} (over $A$) if all generic elements of $S$ have the same Galois type (over $A$).
\end{definition}

\begin{definition}
Let $G$ be a group, and let $q$ be a type.
We say $G$ \emph{acts generically} on the realizations of $q$, if the action $\sigma(u)$ is defined whenever $\sigma \in G$ and $u$ is a realization of $q$ that is generic over $\sigma$.
\end{definition}

We can now state the main theorem of this section.
We will prove it as a series of lemmas (thus, Lemmas \ref{interdefinable}-\ref{tulolemma} are part of the proof of the theorem).
 
\begin{theorem}\label{groupmain}
Suppose $A$ is a finite set, $(a,b,c,x,y,z)$ is a strict bounded partial quadrangle over $A$ and $t(a,b,c,x,y,z/A)$ is stationary.
Then, there is a group $G$ in $\M^{eq}$, Galois definable over some finite set $A' \subset \M$.
Moreover, $G$ has unique generics, and a generic element of $G$ has dimension $n$.

There is a stationary type $q$ such that $G$ acts generically on the realizations of $q$.
If $\sigma, \tau \in G$ and there is some $u$ realizing $q \vert_{\sigma, \tau}$ such that $\sigma(u)=\tau(u)$, then $\sigma=\tau$.
\end{theorem}

\begin{proof} 
As in \cite{lisuriart} (proof of Theorem 3.9), we may without loss assume $A=\emptyset$.
We begin our proof by replacing 
the tuple $(a,b,c,x,y,z)$ with one boundedly equivalent with it so that $z$ and $y$ become interdefinable over $b$.
  
For each $n$ we first define an equivalence relation $E^n$ on $\M^n$ so that $uE^nv$ if and only if $\textrm{bcl}(u)=\textrm{bcl}(v)$.
Similarly, define an equivalence relation $E^{*}$ on $\M^{eq}$ so that $uE^{*}v$ if and only if $\textrm{bcl}(u)=\textrm{bcl}(v)$ (in $\M^{eq}$).
By Lemma 3.10 in \cite{lisuriart}, the element $u/E^n$ is interdefinable with $(u/E^n)/E^*$.
Moreover, $dim(u)=dim(u/E^n)=dim((u/E^n)/E^*).$
Replace now $x$ with $x/E^n$, $y$ with $y/E^n$ and $z$ with $z/E^n$.
The new elements are interbounded with the old ones, so we still have a strict bounded partial quadrangle. 
From now on, denote this new $6$-tuple by $(a,b,c,x,y,z)$.
  
Let $a' \in \M$ be such that $Lt(a'/b,z,y)=Lt(a/b,z,y)$
and $a' \da  abcxyz$.  
Then, there are tuples $c',x'$ such that $\textrm{Lt}(a',c',x'/ b,z,y)=Lt(a,c,x/b,z,y)$, and  $(a',b,c',x',y,z)$ is a strict bounded partial quadrangle over $\emptyset$.
Similarly, we find an element $c'' \in \M$ such that $c'' \da abcxyza'c'x'$  and  elements $a'', x''$ so that $(a'',b,c'',x'',y,z)$ is a strict bounded partial quadrangle over $\emptyset$.
The below picture may help the reader.
\begin{figure}[htbp] \label{fig:multiscale}
\begin{center}
   \includegraphics[scale=0.9]{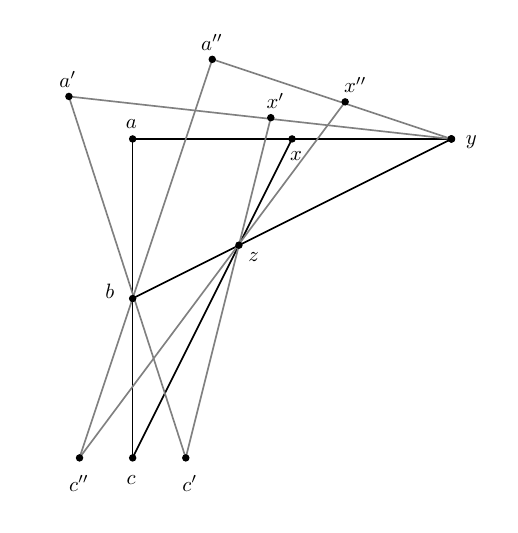}
\end{center}
\end{figure}

We will add the elements $a'$ and $c''$ as parameters in our language, but this will affect the closure operator and the independence notion.
In our arguments, we will be doing calculations both in the set-up we have before adding these parameters and the one obtained after adding them.
We will use the notation cl and $\da$ for the setup before adding the parameters, and $\textrm{cl}^*$ and $\da^*$ for the setup after adding the parameters, i.e. for any sets $B,C,D$, $\textrm{cl}^*(B)=\textrm{cl}(B,a',c'')$ and 
$B \da^{*}_C D$ if and only if $B \da_{Ca' c''} D$.
Similarly, we write $u \in \textrm{dcl}^*(B)$ if and only if $u \in \textrm{dcl}(Ba'c'')$ and use the notation $\textrm{Cb}^*(u/B)$ for $\textrm{Cb}(u/Ba'c'')$.

\begin{lemma}\label{interdefinable}
The tuples $yx'$ and $zx''$ are interdefinable over $a''bc'$ in $\mathbb{M}^{eq}$ after adding the parameters $a'$ and $c''$ to the language.
\end{lemma}
 
\begin{proof}
As in \cite{lisuriart}, Claim 3.12, we see that if for some $z'$, it holds that $t(z'/byc'x')=t(z/byc'x')$, then both $z$ and $z'$ are interbounded (with respect to bcl) with $Cb(b,y/c',x')$ and thus with each other.
It follows that $u=z/E^*$ if and only if there is some $w$ such that $t(w/byc'x')=t(z/byc'x')$ and $w/E^*=u$.
From this, it follows that $z/E^* \in \textrm{dcl}(byc'x')$.
In the beginning of the proof, we replaced $z$ by $z/E^n$, and by Lemma 3.10 in \cite{lisuriart}, it is interdefinable with $z/E^*$.
Thus, $z \in \textrm{dcl}(byc'x')$.
  
For $zx'' \in \textrm{dcl}^{*}(a''bc'yx')$, it suffices to show that $x'' \in \textrm{dcl}^{*}(a''bc'yx'z)$.
But as above, if there is some $x^*$ such that $t(x^*/a''yzc'')=t(x''/a''yzc'')$, then $x^*$ and $x'$ are interbounded, and thus $x''/E^* \in \textrm{dcl}(a''c''yz) \subseteq \textrm{dcl}^{*}(a''bc'x'yz)$ (note that $\textrm{dcl}^*$ is defined with $c''$ as a parameter), and hence by Lemma 3.10 in \cite{lisuriart},  $x'' \in \textrm{dcl}^*(a''bc'x'yz)$.
Similarly, one proves that $yx' \in \textrm{dcl}^*(a''bc'zx'')$.
\end{proof}

Let $q_1=t(yx'/a'c'')$, $q_2=t(zx''/a'c'')$.
We will consider $Cb(yx',zx''/a''bc')$ as a function from $q_1$ to $q_2$.
To see precisely how this is done, we need to introduce some concepts.
 
\begin{definition} 
Suppose $p$ and $q$ are stationary types over some set $B$.
By a \emph{germ of an invertible definable function} from $p$ to $q$, we mean a Lascar type $r(u,v)$ over some finite set $C$ containing $B$, such that 
if $(x,y)$ realizes $r \vert_D$, then $x$ realizes $p \vert_D$, $y$ realizes $q \vert_D$, $x \in \textrm{dcl}(y,D)$, and $y \in \textrm{dcl}(x,D)$.
\end{definition} 
 
We will denote germs of functions by the Greek letters $\sigma, \tau$, etc.
We note that the germs can be represented by elements in $\M^{eq}$. Just represent the germ determined as above by some Lascar type $r$, by some canonical base of $r$.
If $\sigma$ is this germ and $u$ realizes $p \vert_\sigma$, then $\sigma(u)$ is the unique element $v$ such that $(u,v)$ realizes $r \vert_\sigma$.
Note that if $a$ realizes $p\vert_B$ and $\sigma \in B$, then $\sigma(a)$ realizes $q \vert_\sigma$, and as $\sigma(a) \da_\sigma B$, the element $\sigma(a)$ realizes $q \vert_B$.

We note that the germs can be composed.
Suppose $q'$ is another stationary type over $B$, $\sigma$ is a germ from $p$ to $q$ and $\tau$ is a germ from $q$ to $q'$.
Then, by $\tau.\sigma$ we denote a germ from $p$ to $q'$ determined as follows.
Let $u$ realize $p \vert_{\sigma, \tau}$.
Then, we may think of $\tau.\sigma$ as some canonical base of $Lt((u, \tau(\sigma(u)))/\sigma, \tau)$.
We note that $t(u, \tau(\sigma(a))/B\sigma \tau)$ is stationary since $t(u/B \sigma \tau)$ is stationary as a free extension of a stationary type and since $\tau(\sigma(u))$ is definable from $u$, $\sigma$ and $\tau$.
Thus, $\tau.\sigma \in \textrm{dcl}(\sigma, \tau)$, and the notation is meaningful.
 
We will do a small trick that makes the types $q_1$ and $q_2$ stationary, which will allow us to apply the above methods and consider $Cb(yx',zx''/a''bc')$ as a germ of an invertible definable function from $q_1$ to $q_2$.
First we show that we may without loss suppose $b=Cb(yx',zx''/a''bc')$.
Then, after making the types stationary, we will prove that for independent  $b_1, b_2$ realizing $\textrm{tp}(b/a'c'')$, the composition $b_1^{-1}.b_2$ is a germ of an invertible definable function from $q_1$ to $q_1$.  

Note that since $a'' \in \textrm{bcl}(bc'') \subseteq \textrm{bcl}^{*}(b)$ and $c' \in \textrm{bcl}(a'b) \subseteq \textrm{bcl}^{*}(b)$,
 we have $\textrm{Cb}(yx',zx''/a''bc')=\textrm{Cb}^*(yx',zx''/b)$.
Thus, from Lemma \ref{interdefinable}, it follows that the tuples $yx'$ and $zx''$ are interdefinable over $\textrm{Cb}^*(yx',zx''/b)$ after adding the parameters $a''c$ to the language.
We will eventually view $\textrm{Cb}^*(yx',zx''/b)$ as a germ of a function taking $yx' \mapsto zx''$.

We claim that after adding the parameters, $b$ is interbounded with $\textrm{Cb}^*(yx',zx''/b)$.
Denote $\alpha=\textrm{Cb}^*(yx',zx''/b)$ and $\alpha'=Cb(yz/b)$.
By the choice of $a'$ and $c''$, we have $yz \da_b a'c''$, so we may assume $\alpha'=Cb^*(yz/b)$.
Then, $\alpha' \in \textrm{bcl}(\alpha)$.
By (v) in Definition \ref{conf}, $b \in \textrm{bcl}(\alpha')$, and the claim follows.
Thus, we may without loss assume that $b=\textrm{Cb}^*(yx',zx''/b)$.

We now do the trick to make $q_1$ and $q_2$ stationary.
To simplify notation, denote for a while $d=(a,b,c,x,y,z,c',x',a'', x'')$.
Choose now a tuple $d' \in \M^{eq}$ such that $Lt(d'/a'c'')=Lt(d/a'c'')$ and $d' \da_{a'c''} d$.
Then, there is some $d'' \in \M$ such that $d'=F(d'')$ for some definable function $F$ (composition of functions of the form $F_E$, where $E$ is some equivalence relation) and $d''  \da_{a'c''} d$.
Now, for any subsequence $e \subseteq d$, the type $t(e/a'c''d'')$ is stationary.
Indeed, there is some subsequence $e' \subset d''$ such that $Lt(F(e')/a'c'')=Lt(e/a'c'')$ for some definable function $F$.
Thus, $t(e/a'c''e')$ (and hence $t(e/a'c''d'')$) determines $Lt(e/a'c'')$. 

We add the tuple $d''$ as parameters to our language.
Since it is independent over $a'c''$ from everything that we will need in the independence calculations that will follow, the calculations won't depend on whether we have added $d''$ or not.
Thus, we may from now without loss assume $d''=\emptyset$ to simplify notation.

Let $r=t(b/a',c'')$, and note that it is stationary since we added $d''$ to the language.
If $b_1, b_2$ realize $r$, then by $b_1^{-1}.b_2$ we mean the germ of the invertible definable function from $q_1$ to $q_1$ obtained by first applying $b_2$, then $b_1^{-1}$.
In other words, let $y_1 x_1'$ realize $q_1 |_{b_1b_2}$, and let $z_1x_1''=b_2.(y_1 x_1')$, a realization of $q_2|_{ b_1b_2}$.
Let $y_2 x_2'=b_1^{-1}.(z_1x_1'')$ (i.e. $z_1x_1''=b_1.(y_2 x_2')$).
We may code the germ $b_1^{-1}.b_2$ by some canonical base of $t(y_1 x_1',y_2 x_2'/b_1,b_2,a',c'')$, i.e.
we will have $b_1^{-1}.b_2=\textrm{Cb}^*(y_1 x_1',y_2 x_2'/b_1,b_2)$.
At this point, we fix the type of this canonical base.
As noted before, we have
$b_1^{-1}.b_2 \in \textrm{dcl}^*(b_1,b_2)$.

\begin{lemma}\label{vapaus}
Let $b_1$, $b_2$ realize $r$ ($=t(b/a'c'')$), and let $b_1 \da^* b_2$. 
Then, $b_1^{-1}.b_2 \da^* b_i$ for $i=1,2$.
In particular, $dim(b_1^{-1}.b_2/a'c'')=n.$
\end{lemma}

\begin{proof}
Without loss of generality, $b_2=b$ and $b_1 \da^* a,b,c,x,y,z,c',x',a'',x''$.

From the choice of the new elements, it can be calculated that $b \da^* cxzx''$ (see \cite{lisuriart} (or \cite{lisuri}), proof of Lemma 3.13 for details).  
By stationarity of $r$, we have  $t(b/a'c''cxzx'')=t(b_1/a'c''cxzx'')$.  
Hence, there are elements $a_1, y_1, c_1',x_1',a_1''$ so that
$$t(a_1,b_1,c,x,y_1,z,c_1',x_1',a_1'',x''/a'c'')=t(a,b,c,x,y,z,c',x',a'',x''/a'c'').$$
To visualize this, think of the picture just before Lemma \ref{interdefinable}.
In the picture, keep the lines $(c,x,z)$ and $(c'',z,x'')$ fixed pointwise and move $b$ to $b_1$ by an automorphism fixing $a'c''$.
As a result, we get another similar picture drawn on top of the first one, with new elements $a_1, y_1, c_1'$ and $a_1''$ in the same configuration with respect to the fixed points as $a,y,c$ and $a''$ in the original picture.

We now present five auxiliary claims needed in the argument. The proofs can be found in \cite{lisuriart} or \cite{lisuri}  (Claims 3.14-3.18). 

\begin{enumerate}[(i)]
\item   
$a a_1 b b_1 \downarrow^* yx'.$ 
\item $y_1 x_1' \in \textrm{bcl}^*(a,a_1,y).$
 
\item $y_1x_1'=(b_1^{-1}.b)(yx')$. 
 
\item $a a_1 \downarrow^* b.$
 
\item
$aa_1 \downarrow^* b_1.$
\end{enumerate}

Denote $\sigma=b_1^{-1}.b$.
By (i) $yx' \downarrow_{aa_1}^* a a_1 b b_1$.
Thus, by (ii),  we get $yx'y_1x_1' \downarrow _{aa_1}^* a a_1 b b_1$.
On the other hand, by (i), $yx' \downarrow_{bb_1}^* a a_1 b b_1$.
By Claim (iii), $y_1x_1' \in \textrm{bcl}^*(yx',b,b_1)$, so $yx'y_1x_1' \downarrow_{bb_1}^* a a_1bb_1$.
Since $\sigma=\textrm{Cb}^*(yx',y_1x_1'/b,b_1)$, we also have 
$$\sigma=\textrm{Cb}^*(yx',y_1x_1'/a,a_1,b,b_1).$$
So, $\sigma \in \textrm{bcl}^*(a,a_1)$ since $yx'y_1x_1' \downarrow _{aa_1}^* a a_1 b b_1$.
By (iv) and (v), $\sigma \downarrow^* b$ and $\sigma \downarrow^* b_1$.

Since $\sigma \in \textrm{bcl}^*(bb_1)$, we have $2n=dim(b_1b\sigma/a'c'')$, and it follows that $dim(\sigma/a'c'')=n$.
\end{proof}

Denote now $\sigma=b_1^{-1}.b_2$ (from Lemma \ref{vapaus}) and let $s=t(\sigma/a'c'')$ (note that $t(\sigma^{-1}/a'c'')=s$ also).

\begin{lemma}\label{ykkonen}
Let $\sigma_1, \sigma_2$ be realizations of $s$ such that $\sigma_1 \da^* \sigma_2$.
Then, $\sigma_1.\sigma_2$ realizes $s|_{\sigma_i}$ for $i=1,2$.
\end{lemma}

\begin{proof}
As 1. in the proof of Lemma 4.8 in \cite{pi} (see \cite{lisuri}, Lemma 3.19, for details). 
\end{proof}

Let $G$ be the group of germs of functions from $q_1$ to $q_1$ generated by $\{\sigma \, | \, \sigma \textrm{ realizes } s\}$ (note that this set is closed under inverses and thus indeed is a group). 

\begin{lemma}\label{tulolemma}
For any $\tau \in G$, there are $\sigma_1, \sigma_2$ realizing $s$ such that $\tau=\sigma_1.\sigma_2$ and $\tau \da^* \sigma_1$. 
\end{lemma}

\begin{proof}
It is enough to show that if $\tau_i$ realize $s$ for $i=1,2,3$, then there are $\sigma_1, \sigma_2$ realizing $s$ so that $\tau_1.\tau_2.\tau_3=\sigma_1.\sigma_2$ and $\tau_1.\tau_2.\tau_3 \da^* \sigma_1$.
Let $\sigma$ realize $s|_{\tau_1 \tau_2 \tau_3}$.
Now, $\sigma^{-1}.\tau_2 \downarrow_{\tau_2}^* \tau_1 \tau_2 \tau_3$.
By Lemma \ref{ykkonen}, $\sigma^{-1}.\tau_2$ realizes $s |_{\tau_2}$, and thus
$\sigma^{-1}.\tau_2 \downarrow^* \tau_1 \tau_2 \tau_3$.
By Lemma \ref{ykkonen}, $(\sigma^{-1}.\tau_2).\tau_3$ and $\tau_1.\sigma$ realize $s$.
Choosing $\sigma_1=\tau_1. \sigma$ and $\sigma_2=\sigma^{-1}.\tau_2.\tau_3$, we get $\sigma_1.\sigma_2=\tau_1.\tau_2.\tau_3$.
Using the choice of $\sigma$ and Lemma \ref{ykkonen}, it is easy to see that $\tau_1.\tau_2.\tau_3 \da^* \sigma_1$.
\end{proof}

Consider the set 
$$G'=\{\sigma_1.\sigma_2 \, | \, \textrm{$\sigma_1, \sigma_2$ are realizations of $s$}\}.$$
It is clearly Galois definable over $a'c''$.
Let $E$ be the equivalence relation such that for $\gamma_1, \gamma_2 \in G'$, we have $(\gamma_1, \gamma_2) \in E$ if and only if  for all (some) $u$ realizing $q_1 \vert_{\gamma_1 \gamma_2}$, it holds that $\gamma_1(u)=\gamma_2(u)$.
Then, $G=G'/E$, and $G$ is  Galois definable over $a'c''$.
 
To complete the proof of Theorem \ref{groupmain}, it remains to show that a generic element of $G$ has dimension $n$ over $a'c''$ and that $G$ has unique generics.
Suppose $\tau \in G$.
By Lemma \ref{tulolemma}, there are $\sigma_1$, $\sigma_2$ realizing $s$ such that $\tau=\sigma_1.\sigma_2$ and $\tau \da^* \sigma_1$.
Now,
$$dim(\sigma_1.\sigma_2/a'c'')=dim(\sigma_1.\sigma_2/\sigma_1 a'c'')\le dim(\sigma_2/\sigma_1 a'c'') \le n,$$
where the first inequality follows from the fact that $\sigma_2 \in \textrm{bcl}^*(\sigma_1, \sigma_1.\sigma_2)$, and the second one from Lemma \ref{vapaus} and the choice of the type $s$.  
It is easy to see that equality holds if and only if $\sigma_1 \da^* \sigma_2$.
Thus, generic elements have dimension $n$ and by Lemma \ref{ykkonen}, they all have the same type. 
This concludes the proof of the theorem.
\end{proof}     

From now on, we will call the configuration given in Definition \ref{conf} the \emph{group configuration}. 

\section{Properties of the group}
 
In this section, we take a closer look at the group obtained from Theorem \ref{groupmain}.
We first discuss non-classical groups, which are defined by the following two definitions from \cite{HLS}: 

\begin{definition}
An infinite group $G$ \emph{carries an $\omega$-homogeneous pregeometry} if there is a closure operator cl on the subsets of $G$ such that $(G, \textrm{cl})$ is a pregeometry and $dim_{cl}(G)=\vert G \vert$, and whenever $A \subset G$ is finite and $a,b \in G \setminus \textrm{cl}(A)$, then there is an automorphism of $G$ preserving cl, fixing $A$ pointwise and sending $a$ to $b$.
\end{definition}

\begin{definition}
We say that a group is \emph{non-classical} if it is non-Abelian and carries an $\omega$-homogeneous pregeometry.
\end{definition}

It is an open question whether non-classical groups exist.
Throughout this paper, we will assume that the monster model $\M$ we are working in does not interpret non-classical groups.  
In this section, we will show that under this assumption, if $n=1$ in Definition \ref{conf}, then the group $G$ obtained from Theorem \ref{groupmain} is Abelian, and if $n=2$, then an algebraically closed field can be interpreted in $\M$. 
The arguments were originally presented in the first-order case by E. Hrushovski, and the non-elementary case is treated in \cite{HLS} and \cite{locmod}.

For $f,g \in G$, we use $f.g$ to denote group multiplication.
However, since writing $f$ and $g$ as a tuple using $fg$ can be confusing, we write $f,g \da x$ when we mean that $fg$ as a tuple is independent from $x$.
If we mean that $f \da x$ and $g \da x$ (but not necessarily $fg \da x$), we will mention it separately.
By Theorem \ref{groupmain}, the group $G$ is Galois definable over some finite set $A'$.
To simplify notation, we will from now on (without loss) assume $A'=\emptyset$. 

\begin{theorem}\label{1abelian}
Let $\M$ be a quasiminimal pregeometry structure that does not interpret a non-classical group, and suppose that $G$ is a Galois definable group interpretable in $\M$.  
Suppose $G$ has unique generic type and the generic elements of $G$ are of dimension $1$ (with respect to the pregeometry induced by the bounded closure operator bcl).
Then, $G$ is Abelian.
\end{theorem}  
   
\begin{proof} 
The bounded closure operator bcl gives a natural pregeometry on $G$ 
(remember that by Theorems 2.75 and 2.70 in \cite{lisuriart}, the pregeometry obtained from bcl extends to $\M^{eq}$ and that the independence calculus is obtained from it also there). 
Since $G$ has unique generic type, this pregeometry is $\omega$-homogeneous.
Thus, $G$ is Abelian by our assumptions. 
\end{proof}   
 
\begin{remark}\label{confremark}
It follows from Theorem \ref{1abelian} that if $\M$ does not interpret a non-classical group and $n=1$ in the group configuration (\ref{conf}), then the group obtained from the configuration by Theorem \ref{groupmain} is Abelian.
\end{remark} 
 
Next, we show that if $n=2$, then $G$ interprets an algebraically closed field.
To apply the method from \cite{HLS}, we need a total, $2$-determined and $2$-transitive action.
Since we only have a generic action, we have to make some modifications to get this.
But first, we take a look at some properties of the generic action.  
In particular, to eventually obtain an action that is as wanted, we will show that the generic action is generically $2$-transitive (Lemma \ref{gentrans}) and preserves the pregeometry (Lemma \ref{pregsail}).

By Theorem \ref{groupmain}, $G$ acts generically on the realizations of some stationary type.
Denote from now on by $D$ the set of realizations of this type.

\begin{lemma}\label{ag(a)its}
Let $g \in G$, $dim(g) \ge 1$, and let $a \in D$ be generic over $g$.
Then, $a \da g(a)$.
\end{lemma}

\begin{proof}
Suppose $g(a) \in \textrm{bcl}(a)$.
We claim that $g \in \textrm{bcl}(a)$, which will yield a contradiction.
Suppose not.
Let $(g_i)_{i<\o_1}$ be distinct realizations of $t(g/a)$.
Then, for each $i$, $g_i \da a$ and $g_i(a) \in \textrm{bcl}(a)$.
By the pigeonhole principle, there must be some $i<j<\o_1$ such that $g_i(a)=g_j(a)$.
By Theorem \ref{groupmain}, $g_i=g_j$, a contradiction.
\end{proof}

\begin{lemma}\label{U4}
Let $g \in G$ be generic, $a, b \in D$ such that $dim(a,b/g)=2$.
Then,  $dim(a,b,g(a),g(b))=4$.
\end{lemma}
 
\begin{proof}
Suppose towards a contradiction that $dim(a,b,g(a), g(b)) \le 3$.
We have $b \da a, g, g(a)$ and by Lemma \ref{ag(a)its},  $a \da g(a)$, so we must have
$g(b) \in \textrm{bcl}(a,b,g(a))$.
We will prove that  $g \in \textrm{bcl}(a, g(a))$, which is a contradiction, since $dim(g)=2$ and $g \da a$.

Suppose not.
Then, there are distinct $(g_i)_{i<\o_1}$ such that $t(g_i/a,g(a))=t(g/a,g(a))$  and $g_i \da_{a,g(a)} b(g_j)_{j<i}$.
Now,  $g_i(b) \in \textrm{bcl}(a,b,g(a))$.
By the pigeonhole principle, there are $j<k<\o_1$ such that $g_k(b)=g_j(b)$.
An easy calculation shows that $b \da g_j, g_k$, and thus 
$g_j=g_k$ by Theorem \ref{groupmain}, a contradiction.
\end{proof}

\begin{lemma}\label{gentrans}
Let $a,b,c,d \in D$ be such that $dim(a,b,c,d)=4$.
Then, there is some $g \in G$ such that $g(a,b)=(c,d)$.
\end{lemma}

\begin{proof}
Let $g' \in G$ be generic.
By Lemma \ref{U4}, there are $a',b' \in D$ such that $dim(a',b',g'(a'), g'(b'))=4$.
Now, there is some $\sigma \in \textrm{Aut}(\M)$ such that \mbox{$\sigma(a',b',g'(a'),g'(b'))=(a,b,c,d)$,} and we may choose $g=\sigma(g')$.
\end{proof}

\begin{lemma}\label{koodimelk}
Suppose $g \in G$, $a,b \in D$, and the set $\{g,a,b\}$ is independent.
Then, $g \in \textrm{bcl}(a,b,g(a), g(b))$.
\end{lemma}

\begin{proof} 
If $dim(g)=0$, this is clear.
If $dim(g)=1$, then it follows from the assumptions and Lemma \ref{ag(a)its} that $g \in \textrm{bcl}(a,g(a))$.
If $dim(g)=2$, the result follows from the assumptions and Lemma \ref{U4}.  
\end{proof}

\begin{lemma}\label{pregsail}
Let $A \subseteq D$, $g \in G$, $g \da a$ for each $a \in A$, and $b \in D$.
Then, $b \in \textrm{bcl}(A)$ if and only if $g(b) \in \textrm{bcl}(\{g(a) \, | \, a \in A\})$.
\end{lemma}

\begin{proof}
It  suffices to show this in case $A$ is finite.
Suppose $A=\{a_1, \ldots, a_n, a_{n+1}, \ldots, a_{m}\}$ and $dim(A)=dim(a_1, \ldots, a_n)=n$.
We will show that $dim(g(a_1), \ldots, g(a_n))=n$.
Since $g^{-1} \da a$ for each $a \in A$, the same argument can then be applied to show that 
$dim(g(A))=n$ implies $dim(A)=n$ 
  
If $g \da a_1, \ldots, a_n, a_{n+1}, \ldots, a_m$, then $dim(g(a_1), \ldots, g(a_m))=n$ since $a_i$ and $g(a_i)$ are interdefinable over $g$.
Suppose now $dim(a_1, \ldots, a_m)=n$, and $g \da a_i$ for each $i$.
Choose $g' \in G$ generic so that $g' \da g, a_1, \ldots, a_m$.
Then, $g'.g \da a_1, \ldots, a_m$, and thus $dim(g'.g(a_1), \ldots, g'.g(a_m))=n$.
Suppose for the sake of contradiction that $dim(g(a_1), \ldots, g(a_m))=k \neq n$.
But now, $dim(g'.g(a_1), \ldots, g'.g(a_m))=dim(g(a_1), \ldots, g(a_m))=k \neq n$, a contradiction.
 \end{proof}
 
Applying lemmas \ref{gentrans}, \ref{koodimelk} and \ref{pregsail}, we can now conclude the following theorem using the arguments of sections 3 and 4 in \cite{locmod}.
  
\begin{theorem}\label{field1}
Let $\M$ be a quasiminimal pregeometry structure, and suppose there is a group configuration in $\M$ with $n=2$.
Then, there is either an algebraically closed field or a non-classical group in $\M^{eq}$.
\end{theorem}
 
From now on, a group configuration with $n=2$ will be called a \emph{field configuration}.
 
\subsection{Obtaining a field from an indiscernible array}

Our goal is to find a field in a non-locally modular Zariski-like structure that does not interpret non-classical groups.
We will use Theorem \ref{field1}, but for our purposes it is practical to reformulate it in terms of indiscernible arrays, as is done in the first-order context in \cite{HrZi}.
 
\begin{definition}
We say that $f=(f_{ij} : i \in I, j \in J)$, where $I$ and $J$ are ordered sets, is an \emph{indiscernible array} over $A$ if whenever $i_1, \ldots, i_n \in I$, $j_1, \ldots, j_m \in J$, $i_1 < \ldots < i_n$, $j_1 <\ldots < j_m$, then $t((f_{i_{\nu} j_{\mu}}: 1 \le \nu \le n, 1 \le \mu, \le m)/A)$ depends only on the numbers $n$ and $m$.

If at least the dimension  of the above sequence depends only on $m,n$, and $dim((f_{i_{\nu} j_{\mu}}: 1 \le \nu \le n, 1 \le \mu \le m)/A)= \alpha(m,n)$, where $\alpha$ is some polynomial of $m$ and $n$, we say that $f$ is \emph{rank-indiscernible}  over $A$, of type $\alpha$, and write $dim(f;n,m/A)=\alpha(n,m)$.
\end{definition}  
 
If $(c_{ij}: i \in I, j \in J)$ is an array and $I' \subseteq I$, $J' \subseteq J$ , we write $c_{I'J'}$ for $(c_{ij}: i \in I', j \in J'$).
If $|I'|=m$ and $|J'|=n$, we call $c_{I'J'}$ an $m \times n$ -\emph{rectangle} from $c_{ij}$.

The following is Lemma 4.15 from \cite{lisuriart} (or \cite{lisuri}).
 
\begin{lemma}\label{samelascar}
Let $f=(f_{ij} : i, j \in \kappa)$ be an indiscernible array over $A$, and let $\kappa \ge \o_1$.
Then, for all $m,n$, all the $m \times n$ rectangles of $f$ have the same Lascar type over $A$. 
\end{lemma}

\begin{lemma}\label{indiscpermu}
Let $(f_{ij} : i \in \o_1, j \in \o_1)$ be an indiscernible array, and let $I_0 \subset I$ and $J_0 \subset J$ be finite.
If $\pi$ is a permutation of $I_0$ and $\pi'$ is a permutation of $J_0$, then 
$$t((f_{ij}:i \in I_0, j \in J_0)/\emptyset)=t((f_{ij}:i \in \pi(I_0), j \in \pi'(J_0))/\emptyset).$$ 
\end{lemma}

\begin{proof}
We show first that $t((f_{ij}:i \in I_0, j \in J_0)/\emptyset)=t((f_{ij}:i \in \pi(I_0), j \in J_0)/\emptyset).$
Consider the sequence $(f_{ij}: j \in J_{0})_{i \in \o_1}$.
By Lemmas 2.24 and 2.26 in \cite{lisuriart}, there is some cofinal $X \subset \o_1$ such that every permutation of the sequence $(f_{ij}: j \in J_{0})_{i \in X}$ extends to an automorphism of $\M$.
Since $f$ is indiscernible, every order-preserving function from $I_0$ to $X$ extends to an automorphism.
Thus, we can first send $I_0$ into $X$ using some automorphism $g$, then apply the automorphism corresponding to $\pi$ and finally take the permuted sequence back using $g^{-1}$. 

We can now repeat the argument with $j$ in place of $i$ to prove the Lemma.
\end{proof}

The following lemma formulates Theorem \ref{field1} in terms of indiscernible arrays.
 
 \begin{lemma}\label{groupexists}
 Let $\M$ be a quasiminimal pregeometry structure, and suppose there is a finite set $B$ and an indiscernible array 
 $f=(f_{ij} : i,j< \omega_1)$ of elements of $\M$ of type $\alpha(m,n)=2m+n-2$ for $m,n \ge 2$ such that $dim(f;1,k)=k$ for any $k$.  
Then, either there is an algebraically closed field or a non-classical group in $\M^{eq}$. 
 \end{lemma}
 
\begin{proof}
Like in the proof of Lemma 6.3 in \cite{HrZi}, we can find a field configuration in $\M$, and Theorem \ref{field1} then yields an algebraically closed field.
\end{proof}

\section{Fields in Zariski-like structures}

At the end of this section, we will prove our main result: that an algebraically closed field can be found in a non locally modular Zariski-like structure that doesn't interpret a non-classical group (Theorem \ref{MAIN}).
Here we adapt the ideas behind the proof of Lemma 6.11 in \cite{HrZi} to our setup.
 
Before we can list the axioms for a Zariski-like structure, we need some auxiliary definitions.
First, we generalize the notion of specialization from \cite{HrZi} to our context.

\begin{definition}\label{speciso}
Let $\M$ be a monster model for a quasiminimal class, $A \subset \M$, and let $\mathcal{C}$ be a collection of subsets of $\M^n$, $n=1,2,\ldots$.
We say that a function $f: A \to \M$ is a \emph{specialization} (with respect to $\mathcal{C}$) if for any $a_1, \ldots, a_n \in A$ and for any $C \in \mathcal{C}$, it holds that if $(a_1, \ldots, a_n) \in C$, then $(f(a_1), \ldots, f(a_n)) \in C$.
If $A=(a_i : i \in I)$, $B=(b_i:i\in I)$ and the indexing is clear from the context, we write $A \to B$ if the map $a_i \mapsto b_i$, $i \in I$, is a specialization.

We say that the specialization $f$ is an \emph{isomorphism} if also the converse holds, that is if $C \in \mathcal{C}$ and $(f(a_1), \ldots, f(a_n)) \in C$, then $(a_1, \ldots, a_n) \in C$.

If $a$ and $b$ are finite tuples and $a \to b$, we denote $\textrm{rk}(a \to b)=\textrm{dim}(a/\emptyset)-\textrm{dim}(b/\emptyset)$.
\end{definition}

The specializations in the context of Zariski geometries in \cite{HrZi} are specializations in the sense of our definition if we take $\mathcal{C}$ to be the collection of closed sets (Zariski geometries are quasiminimal since they are strongly minimal).
 
Next, we generalize the definition of regular specializations from \cite{HrZi}.  

\begin{definition}
Let $\M$ be a monster model for a quasiminimal class.
We define a \emph{strongly regular} specialization recursively as follows:
\begin{itemize}
\item Isomorphisms are strongly regular;
\item If $a \to a'$ is a specialization and $a \in \M$ is generic over $\emptyset$, then $a \to a'$ is strongly regular;
\item $aa' \to bb'$ is strongly regular if $a \downarrow_\emptyset a'$ and the specializations $a \to b$ and $a' \to b'$ are strongly regular.
\end{itemize}
\end{definition}

\begin{remark}
It follows from Assumptions 6.6 (7) in \cite{HrZi} (for a more detailed discussion on why these properties hold in a Zariski geometry, see \cite{lisuri}, Chapter 1.1.) that if a specialization on a Zariski geometry is strongly regular in the sense of our definition, then it is regular in the sense of \cite{HrZi} (definition on p. 25).
\end{remark}
  
The following generalizes the definition of good specializations from \cite{HrZi}.

\begin{definition}\label{good}
We define a \emph{strongly good} specialization recursively as follows.
Any strongly regular specialization is strongly good.
Let $a=(a_1, a_2, a_3)$, $a'=(a_1', a_2', a_3')$, and $a \to a'$.
Suppose:
\begin{enumerate}[(i)]
\item $(a_1, a_2) \to (a_1', a_2')$ is strongly good.
\item $a_1 \to a_1'$ is an isomorphism.
\item $a_3 \in \textrm{cl}(a_1)$.
\end{enumerate}
Then, $a \to a'$ is strongly good.
\end{definition}

A Zariski-like structure is defined by nine axioms as follows.  

\begin{definition}
We say that a quasiminimal pregeometry structure $\M$ is \emph{Zariski-like} if there is a countable collection $\mathcal{C}$ of subsets of $\M^n$ ($n=1, 2, \ldots$), which we call the \emph{irreducible} sets, satisfying the following axioms (all specializations are with respect to $\mathcal{C}$).
 
\begin{enumerate}[(1)]
\item Each $C \in \mathcal{C}$ is Galois definable over $\emptyset$.
\item For each $a \in \M$, there is some $C \in \mathcal{C}$ such that $a$ is generic in $C$.
\item If $C \in \mathcal{C}$, then the generic elements of $C$ have the same Galois type. 
\item If $C,D \in \mathcal{C}$, $a \in C$ is generic and $a \in D$, then $C \subseteq D.$
\item If $C_1, C_2 \in \mathcal{C}$, $(a,b) \in C_1$ is generic, $a$ is a generic element of $C_2$ and $(a',b') \in C_1$, then $a' \in C_2$.
\item If $C \in \mathcal{C}$, $C \subset \M^n$, and $f$ is a coordinate permutation on $\M^n$, then $f(C) \in \mathcal{C}$.
\item Let $a \to a'$ be a strongly good specialization  and let $rk(a \to a') \le 1$.
Then any specializations $ab \to a'b'$, $ac \to a'c'$  can be amalgamated: there exists $b^{*}$, independent from $c$ over $a$ such that $\textrm{t}^g(b^{*}/a)=\textrm{t}^g(b/a)$, and $ab^{*}c \to a'b'c'$.
\item Let $(a_i:i\in I)$ be independent and indiscernible over $b$.
Suppose $(a_i':i\in I)$ is indiscernible over $b'$, and $a_ib \to a_i'b'$ for each $i \in I$.
Further suppose $(b \to b')$ is a strongly good specialization  and $rk(b \to b') \le 1$.
Then, \mbox{$(ba_i:i \in I) \to (b'a_i':i\in I)$.}

\item Let $\kappa$ be a (possibly finite) cardinal and let  $a_i, b_i \in \M$ with $i < \kappa$, such that $a_0 \neq a_1$ and $b_0=b_1$. 
Denote by $\mathcal{P}_{<\omega}(\kappa)$ the set of all finite subsets of $\kappa$.
Suppose $(a_i)_{i<\kappa} \to (b_i)_{i <\kappa}$ is a specialization.
Assume there is some unbounded and directed $S \subset \mathcal{P}_{<\omega}(\kappa)$ satisfying the following conditions:
\begin{enumerate}[(i)]
\item  $0,1 \in X$ for all $X \in S$;
\item For all $X,Y \in S$ such that $X \subseteq Y$,  and for all sequences $(c_i)_{i \in Y}$ from $\M$, the following holds: 
If $c_0=c_1$, $ (a_i)_{i \in Y} \to (c_i)_{i \in Y} \to (b_i)_{i \in Y},$
and $\textrm{rk}((a_i)_{i \in Y} \to (c_i)_{i \in Y}) \le 1$, then $\textrm{rk}((a_i)_{i \in X} \to (c_i)_{i \in X}) \le 1$.
\end{enumerate}
 
Then, there are $(c_i)_{i<\kappa}$ such that   
$$(a_i)_{i < \kappa} \to (c_i)_{i < \kappa} \to (b_i)_{i < \kappa},$$
$c_0=c_1$ and $\textrm{rk}((a_i)_{i \in X} \to (c_i)_{i \in X}) \le 1$ for all $X \in S$.
 \end{enumerate}
\end{definition}
 
\begin{definition}
Let $\M$ be Zariski-like and let $a \in \M$.
By axioms (2) and (4), there is a unique $C \in \mathcal{C}$ such that $a$ is generic in $C$.
The set $C$ is called the \emph{locus} of $a$.
\end{definition}

\begin{remark}\label{isoinj}
Note that on Zariski-like structures, specializations that are isomorphisms in the sense of Definition \ref{speciso} are injective.
Indeed, suppose that $M$ is Zariski-like, $f$ is a specialization that is an isomorphism, and $f(a)=f(b)$ for some $a,b \in M$.
Then, by Axiom (3), $t^g(a,b/\emptyset)=t^g(f(a),f(b)/\emptyset)=t^g(f(a),f(a)/\emptyset)$, so $a=b$.
\end{remark}

\begin{remark}
Zariski geometries are Zariski-like if the collection $\mathcal{C}$ is taken to be the irreducible closed sets definable over $\emptyset$.
Indeed, strongly minimal structures are quasiminimal, and the conditions (1)-(9) are satisfied.
On a Zariski geometry, first-order types imply Galois types.
Moreover, every strongly regular specialization is regular, and every strongly good specialization is good.
Hence, (7) is Lemma 5.14 in \cite{HrZi} and (8) is Lemma 5.15 in \cite{HrZi}.
For (9), we note that if the assumptions hold, then by the Dimension Theorem of Zariski geometries (Lemma 4.13 in \cite{HrZi}), 
the conclusion holds for finite subsets of $S$, and thus (9) holds by Compactness.
Indeed, if $S_0 \subset S$ is finite, we may, since $S$ is directed, without loss assume that some $Y \in S_0$ contains all other members of $S_0$ as subsets.
By the Dimension Theorem, there are $(c_i)_{i \in Y}$ such that $(a_i)_{i \in Y} \to (c_i)_{i \in Y} \to (b_i)_{i \in Y}$, $c_0=c_1$, and $rk( (a_i)_{i \in Y} \to (c_i)_{i \in Y}) \le 1$. It follows from condition (ii) that $rk( (a_i)_{i \in X} \to (c_i)_{i \in X}) \le 1$ for all $X \in S_0$.

The cover of the multiplicative group of an algebraically closed field of characteristic zero, studied in e.g. \cite{Zilber}, is a non-trivial example of a Zariski-like structure.
On the cover, a topology can be introduced by taking the sets definable by positive, quantifier-free formulae as the basic closed sets, as is done in \cite{Lucy}.
In \cite{coverart}, it is shown that if $\mathcal{C}$ is taken to consist of the irreducible, closed sets that are definable over the empty set (after adding countably many symbols to the language), then the axioms are satisfied.
\end{remark}

\begin{remark}
We note that in \cite{HrZi}, the counterpart of (8) (Lemma 5.15) is deduced from the counterpart of (7) (Lemma 5.14).
However, the argument uses the fact that in the first order strongly minimal setting, there are only finitely many strong types over a finite set.
Thus, it cannot be adapted to our setting where we can have (countably) infinitely many Lascar types over a finite set,
so we have taken (8) as an axiom.
\end{remark}

Next, we note that Axiom (9) of Zariski-like structures implies the usual dimension theorem of Zariski geometries (in the form of Lemma 4.13 in \cite{HrZi}).

\begin{lemma}\label{dimthm}
Let $\M$ be a Zariski-like structure, and suppose $a=(a_0, \ldots, a_{n-1})$ and $b=(b_0, \ldots, b_{n-1})$ are such that $a \to b$, $a_0 \neq a_1$ and $b_0=b_1$.
Then, there is some $c=(c_0, \ldots, c_n)$ such that $a \to c \to b$, $rk(a\to c) \le 1$ and $c_0=c_1$.
\end{lemma}

\begin{proof}
Taking $\kappa=n$ and $S=\{n\}$, the conditions  (i) and (ii) of Axiom 9 hold trivially.
\end{proof}

In Zariski geometries, a specialization either drops the rank of a tuple or preserves its (first order) type.
The analogue holds in our setting.

\begin{remark}\label{specciremark}
From Axiom (3), it follows that if $a \to b$ is a specialization, then either $dim(b)<dim(a)$ or $t^g(a/\emptyset)=t^g(b/\emptyset)$.
\end{remark} 

\begin{lemma}\label{finarray}
Let $A$ and $A'$ be finite arrays.
Let $A_{11}=(a,b)$ and $A_{11}'=(a',b')$ for some tuples $a, b, a',b' \in \M$ such that $a \neq a'$.  
Suppose that for some $c \in \M$ it holds that $a'cA \to a'cA'$.

Then, there is an array $A''$ and a tuple $b'' \in \M$ such that $A_{11}''=(a',b'')$, $dim(A''/a'c)=dim(A/a'c)-1$, and $a'cA \to a'cA'' \to a'cA'$.
\end{lemma}

\begin{proof}
After suitably rearranging the indices, we may think of $a'cA$ and $a'cA'$ as tuples $(a',a, \ldots)$ and $(a',a', \ldots)$, respectively.
By Lemma \ref{dimthm}, there is an array $A^*$ and some $a^*, b^*, c^* \in \M$ such that $A^*_{11}=(a^*,b^*)$, that $a'cA \to a^*c^*A^* \to a'cA'$, and $dim(a^*c^*A^*)\ge dim(a'cA)-1$.
By our assumptions and Remark \ref{specciremark}, equality holds.
Since $a'c \to a^*c^* \to a'c$, there is by Axiom (3) of Zariski-like structures, some automorphism $f$ such that $f(a^*c^*)=a'c$.
Then, $f(A^*)$ is as wanted (taking $b''=f(b^*))$.
\end{proof}
Our main theorem (\ref{MAIN})
will state that a Zariski-like structure with a non locally modular canonical pregeometry interprets either an algebraically closed field or a non-classical group.
It is an analogue of Lemma 6.11 in \cite{HrZi}, and many of the details of the proof are direct adaptations from \cite{HrZi}.
There, the authors first deduce the existence of a $1$-dimensional Abelian group $H$, and then use its elements and a family of plane curves of rank at least $2$ to construct an infinite indiscernible array of type $2m+n-2$, which then implies the existence of the field.

Their argument goes roughly as follows.
Suppose $C$ is a family of plane curves of rank at least $2$. 
Take $a_i$, $b_i$, $b^j$, $i,j<\omega$, to be an independent sequence of generic elements of the group $H$, and let $b_{ij}=b_i+b^j$.
For each pair $(i,j)$ there is a curve in $C$ that passes through the points $(a_0, b_{00})$ and 
$(a_i, b_{ij})$.
This curve is parameterized by some $e_{ij}$. 
For each finite number $N$, let $A$ be an $N \times N$ array such that $A_{ij}=(a_i, b_{ij}, e_{ij})$.
Using the fact that the family $C$ has rank at least $2$, it can be calculated that the array is of type $2m+n-1$ (over $a_0 b_{00}$).
It is then shown that there is a specialization from $A$ to a certain array $A''$, and the
dimension theorem, together with a technical lemma (6.9 in \cite{HrZi}), is applied to this specialization to obtain an $N \times N$ -array $A^*$ of type $2m+n-2$.
Finally, Compactness yields an infinite array of type $2m+n-2$.

The main difference in our setting is that we don't have Compactness.
Moreover, since we work with bounded closures rather than algebraic closures,
we need our array to be uncountable, not just infinite.
Hence, we start by constructing a very large (not necessarily indiscernible) array and use a combinatorial trick due to Shelah to obtain an indiscernible array $A$ of size $\omega_1 \times \omega_1$.

By \cite{lisuriart}, a 1-dimensional group exists also in our setting, and it is Abelian assuming that the structure does not interpret non-classical groups.
Both here and in \cite{HrZi}, the group is in $\M^{eq}$ rather than in $\M$.
In \cite{HrZi}, the problem of elimination of imaginaries is tackled using special sorts (see \cite{HrZi}, p. 18).
We take a somewhat different route and ``code" the group elements using tuples in $\M$ (see Lemma \ref{koodaus}).
When talking about a coding, it is usually expected that elements have unique codes.
However, this will not hold in our case, so we speak about \emph{weak codes}.
In the proof, the group elements are used to make certain relationships hold inside the arrays,
and thus they are mainly used for dimension calculations. 
Our coding preserves dimensions and independence calculations (over certain parameters), and hence suffices for the proof.
An approach similar to that in \cite{HrZi} would probably work also in our setting,
but we found it more convenient to use the weak codes.

Using the codes for the group elements, 
we follow \cite{HrZi} in constructing arrays $A$ and $A''$ such that $A \to A''$. 
Further adapting the argument, we then apply the dimension theorem and an analogue of Lemma 6.9 in \cite{HrZi} (our Lemma \ref{tekn2})
to finite subarrays to obtain finite arrays $A^*$ that are rank indiscernible of type $2m+n-2$.
At the end of the proof, we use the properties of these arrays to deduce the existence of an infinite array from Axiom (9) of Zariski-like structures
(here, the sets in the collection $S$ are taken to be the finite subarrays).
This results in an infinite array of type $2m+n-2$ (the sequence $(c_i)_{i<\kappa}$  
of Axiom (9)). 
Finally, we apply the Shelah trick again to make sure that the array obtained is indiscernible (this can be done if all the arrays in the beginning are constructed to be large enough).
Thus, Axiom (9) works as our analogue for Compactness.  
  
In \cite{HrZi}, the fact that specializations preserve the group addition (6.6.(6)) is used to deduce that the array $A^*$ has all the properties necessary to apply Lemma 6.9.
Here, too, we take a different route, and instead of proving an analogue for 6.6(6) use some auxiliary group elements to make sure that all the necessary relations hold.
The result needed for this is captured in our Lemma \ref{binterbound}.  
This is the source of most of the differences in technicalities between the proofs in the two settings. 
For example, to include codes for these auxiliary elements into the specializations between the arrays, 
we start by constructing an array $B$ where every entry contains certain auxiliary elements in addition to the entries of $A$.
We expect that an analogue for 6.6(6) of \cite{HrZi} could probably be proved in our setting, 
but again, we find it more convenient to take a different path.

In \ref{families}, we will present our analogues for the first order notions of plane curves and families thereof.
In \ref{cooding}, we will discuss the weak coding, and in \ref{prooof}, we will prove the main theorem.
 
We now start working towards it with a technical lemma that is the analogue of Lemma 6.9 in \cite{HrZi}.
The proof is a straightforward adaptation, but we present it here since \cite{HrZi} does not give much detail.
It will be applied in the proof of our main theorem to obtain an uncountable indiscernible array of type $2m+n-2$.
The existence of the field will then follow from Lemma \ref{groupexists}.
 
 \begin{lemma}\label{tekn2}
Let $(A_{ij} : 1 \le i \le M, 1 \le j \le N)$ be a subarray of an indiscernible array of size $\o_1 \times \o_1$, over some finite tuple $b$.
Assume $dim(A;m,n/b)=2m+n-1$ and that
\begin{displaymath}
dim(\textrm{dcl}(A_{11} A_{12} A_{13}b) \cap \textrm{dcl}(A_{21} A_{22} A_{23}b)/b)=2.
 \end{displaymath}
 
Let $b(A_{ij} : i \le M, j \le N) \to b(a_{ij} : i \le M, j \le N)$ be a rank-$1$ specialization, and suppose $\textrm{Lt}(a_{ij}/b)=\textrm{Lt}(a_{i'j'}/b)$ for all $i,i' \le M, j,j' \le N$ and $dim(a;1,k/b)=k$.
Further assume $bA_{ij} A_{ij'} A_{ij''} \to ba_{ij} a_{ij'} a_{ij''}$ is strongly good for any $i, j, j', j''$.
Then $a$ is rank-indiscernible, of type $2m+n-2$ ($m \ge 2$, $n \ge 2$) over $b$.  
\end{lemma}
 
\begin{proof}
To simplify notation, we assume $b=\emptyset$.
All the arguments are similar in the general case.

We first note that any $m \times n$ -rectangle from $a$ has dimension at most $2m+n-2$.
Otherwise the specialization $A \to a$ would be an isomorphism on some rectangle and hence on each of its elements. But for each pair $i,j$, $dim(A_{ij})=2$, whereas $dim(a_{ij})=1$.

We will now prove that an $m \times n$ -rectangle actually has dimension at least $2m+n-2$.
We do this as a series of auxiliary claims.

\begin{claim}\label{permutod}
Let $(c_{ij})$ be a $2 \times 3$ -subarray of $a$.  
Assume $(A_{ij} : i=1,2, j=1,2,3) \to (c_{ij} : i=1,2, j=1,2,3)$ is a rank-$1$ specialization.
Then, \mbox{$dim(c_{13} c_{23} /c_{11} c_{21}  c_{12} c_{22})<2$.}
\end{claim}

\begin{proof}
Suppose $dim(c_{13} c_{23} /c_{11} c_{21}  c_{12} c_{22})=2$.
By the type of the array $A$, we have $dim(A; 2,3)=6$ and since we have a rank-$1$ specialization, $dim(c; 2,3)=5$.  Thus, $dim(c_{11} c_{21}  c_{12} c_{22})=3$.
Since we assumed   $dim(a;1,k/b)=k$, we have $dim(c_{21} c_{22})=2$, and thus either $dim(c_{11} c_{21} c_{22})=3$ or $dim(c_{12} c_{21} c_{22})=3$.
In the first case, we let $c_{01}$, $c_{02}$, $c_{03}$ be new elements such that
\begin{displaymath}
Lt(c_{01} c_{02} c_{03} c_{11} c_{12} c_{13}/\emptyset)=Lt(c_{23} c_{21} c_{22} c_{13} c_{11} c_{12}/ \emptyset)
\end{displaymath}
and
\begin{displaymath}
c_{01} c_{02} c_{03} \downarrow_{c_{11} c_{12} c_{13}} c_{21} c_{22} c_{23}
\end{displaymath}
(In the case that  $dim(c_{12} c_{21} c_{22})=3$, we choose the elements $c_{01}$, $c_{02}$, $c_{03}$ so that $Lt(c_{01} c_{02} c_{03} c_{11} c_{12} c_{13}/\emptyset)=Lt(c_{23} c_{21} c_{22} c_{13} c_{12} c_{11}/\emptyset)$, otherwise the proof is similar as the case we are handling here.)
 
Now, we have specializations 
\begin{eqnarray}\label{1}
A_{11} A_{12} A_{13} A_{21} A_{22} A_{23} \to c_{11} c_{12} c_{13} c_{21} c_{22} c_{23}
\end{eqnarray}
and 
$$A_{13} A_{11} A_{12} A_{23} A_{21} A_{22} \to c_{11} c_{12} c_{13} c_{01} c_{02} c_{03}.$$
By Lemma \ref{indiscpermu},
$t(A_{11} A_{12} A_{13} A_{21} A_{22} A_{23} /\emptyset)=t(A_{13} A_{11} A_{12} A_{23} A_{21} A_{22}/\emptyset)$
and hence $A_{11} A_{12} A_{13} A_{21} A_{22} A_{23} \to A_{13} A_{11} A_{12} A_{23} A_{21} A_{22}$.
So
\begin{eqnarray}\label{2} 
A_{11} A_{12} A_{13} A_{21} A_{22} A_{23} \to c_{11} c_{12} c_{13} c_{01} c_{02} c_{03}.
\end{eqnarray}
As $A_{13} A_{11} A_{12} \to c_{13} c_{11} c_{12}$ is a strongly good specialization of rank $1$, we see, by applying Axiom (7) to the specializations (\ref{1}) and (\ref{2}), that there exist
$A_{01}$, $A_{02}$, and $A_{03}$ with
\begin{eqnarray}\label{tyyppihyyppi}
t(A_{01} A_{02} A_{03}/ A_{11} A_{12} A_{13}) = t(A_{21} A_{22} A_{23}/A_{11} A_{12} A_{13})
\end{eqnarray}
and
\begin{displaymath}
A_{01} A_{02} A_{03} \downarrow_{A_{11} A_{12} A_{13}} A_{23} A_{21} A_{22}
\end{displaymath}
such that $(A_{ij} : i=0,1,2, j=1,2,3) \to (c_{ij} : i=0,1,2, j=1,2,3)$, and in particular, $(A_{ij} : i=0,2, j=1,2, 3) \to (c_{ij} : i=0,2, j=1,2,3)$.
We will prove that this specialization is an isomorphism and get a contradiction (as $dim(A_{21} A_{22})=3$ but $dim(c_{21} c_{22})=2$).

We show that $dim((A_{ij} : i=0,2, j=1,2, 3))=6$.
Now $dim(A_{01} A_{02} A_{03} / A_{11} A_{12} A_{13})=dim(A_{21} A_{22} A_{23} / A_{11} A_{12} A_{13})=2$.
Denote $X= \textrm{dcl}(A_{21} A_{22} A_{23}) \cap \textrm{dcl}(A_{11} A_{12} A_{13})$.
By our assumptions and the type of the array $A$, we have $dim(A_{21} A_{22} A_{23} /X)=2$, and since $X \subseteq \textrm{dcl}(A_{11} A_{12} A_{13})$, (\ref{tyyppihyyppi}) implies that also
$dim(A_{01} A_{02} A_{03}/X)=2$.
Thus,
\begin{displaymath}
A_{01} A_{02} A_{03} \downarrow_{X} A_{11} A_{12} A_{13} A_{21} A_{22} A_{23},
\end{displaymath}
and  thus $dim(A_{01} A_{02} A_{03}/A_{21} A_{22} A_{23})=2$.
Hence, \mbox{$dim((A_{ij} : i=0,2, j=1,2, 3))=6$.}

Next, we show that $dim((c_{ij} : i=0,2, j=1,2, 3))=6$.
The counterassumption gives us 
\begin{equation}\label{vastolseur}
c_{13} \downarrow c_{11} c_{12} c_{21} c_{22} c_{23}.
\end{equation}
Hence, 
\begin{equation}\label{perus}
c_{01} c_{02} c_{03} \downarrow_{c_{11} c_{12} } c_{21} c_{22} c_{23}.
\end{equation}
Similarly, 
\begin{equation}\label{uusiyht}
c_{23} \downarrow c_{11} c_{12}  c_{13} c_{21} c_{22}, 
\end{equation}
and thus 
\begin{displaymath}
1=dim(c_{23}/c_{13} c_{11})=dim(c_{01}/c_{11} c_{12})=dim(c_{01}/c_{21} c_{22} c_{23}),
\end{displaymath}
 so $dim(c_{01} c_{21} c_{22} c_{23})=4$.
 We will show that $dim(c_{02}c_{03}/c_{01} c_{21} c_{22} c_{23})=2$, which will give the contradiction.
 From (\ref{perus}) it follows that $c_{02} c_{03} \downarrow_{c_{11} c_{12} c_{01}} c_{21} c_{22} c_{23}$.
 By (\ref{uusiyht}) and the definition of $c_{01}$, we have
 $c_{01} \downarrow c_{11} c_{12} c_{02} c_{03}$, and in particular $c_{02} c_{03} \downarrow_{c_{11} c_{12}} c_{01}$.
Thus, by transitivity, 
\begin{equation}\label{perus2}
c_{02} c_{03} \downarrow_{c_{11} c_{12}} c_{01} c_{21} c_{22} c_{23}.
\end{equation}
From (\ref{vastolseur}) it follows that $c_{13} \downarrow c_{21} c_{22} c_{11}$, hence (by the choice of $c_{02}$ and $c_{03}$) $c_{11} \downarrow c_{02} c_{03} c_{12}$, and in particular 
\begin{displaymath}
c_{11} \downarrow_{c_{12}} c_{02} c_{03}.
\end{displaymath}
From this and (\ref{perus2}) we get
\begin{equation}\label{perus3}
c_{02} c_{03} \downarrow_{c_{12}} c_{01} c_{21} c_{22} c_{23}.
\end{equation}
But $dim(c_{12} c_{02} c_{03})=dim(c_{11} c_{21} c_{22})$, and in the beginning of the proof of the claim, we assumed this rank is $3$, so $dim(c_{02} c_{03}/c_{12})=2$.
Thus, $c_{02} c_{03} \downarrow c_{12}$, and from this and (\ref{perus3}) we finally get
\begin{displaymath} 
dim(c_{02} c_{03}/c_{01} c_{21} c_{22} c_{23})=2,
\end{displaymath}
as wanted. 
\end{proof}

\begin{claim}\label{indices1}
For any set $*$ of $j$-indices and $i \ge 2$, $dim(a_{i,*}/a_{<i,*}) \le 2$.
\end{claim}

\begin{proof}
If the claim fails, we may take $*$ to be a set of three indices such that $dim(a_{i,*}/a_{<i,*}) =3$.
Since $dim(a; 1, 3)=3$, we would get a $2 \times 3$-rectangle of rank $6$, which is against the observation that any $m \times n$ -rectangle has dimension at most $2m+n-2$.
\end{proof}

\begin{claim}\label{indices2}
For any set $*$ of $i$-indices and $j \ge 3$, $dim(a_{*,j}/a_{*,<j}) \le 1$.
\end{claim}

\begin{proof}
Suppose not. We may take $*$ to consist of two indices, say $*=\{1,2\}$, and then $dim(a_{1j}a_{2j}/a_{1,<j}, a_{2, <j})=2$ for some $j$. 
We first observe that there cannot be two distinct values of $j$ for which the claim fails.
Indeed, suppose  $j_1<j_2$ are the two smallest failing indices. 
Then, $dim(a_{1j_2} a_{2 j_2}/a_{1j_1}, a_{2 j_1}, a_{11}, a_{21})=2$.
Denote $r=dim(a_{1j_2}, a_{2 j_2}, a_{1j_1}, a_{2 j_1}, a_{11}, a_{21})$.
As the claim fails for $j_1$ also, we have $dim(a_{1j_1}, a_{2 j_1}/ a_{11}, a_{21})=2$
Now, $dim(a_{11} a_{21})=1$ or  $dim(a_{11} a_{21})=2$.
In the first case, $r=5$, so  $(A_{ij} : i=1,2, j=1, j_1, j_2) \to (a_{ij} : i=1,2, j=1, j_1, j_2)$ is a rank-$1$ specialization and we are contradicting Claim \ref{permutod}.
In the second case, $r=6$ and $(A_{ij} : i=1,2, j=1, j_1, j_2) \to (a_{ij} : i=1,2, j=1, j_1, j_2)$ is an isomorphism which is also a contradiction ($dim(A_{11} A_{21})=4$ but $dim(a_{11} a_{21}) \le 2)$.

So we have at most one failing index $j$. 
Denote it by $j_0$.
We note first that  $dim(a_{1j} a_{2j}: j \le N) \ge 2+N$, since otherwise we could prove by induction on $m$ (using Claim \ref{indices1}) that $dim(a_{ij} : j \le N, i \le m) < 2m+N-2$. 
Setting $m=M$, this would contradict the assumption that $A \to a$ is a rank-$1$ specialization.

Thus,  
\begin{displaymath}
2+N \le dim(a_{1j}a_{2j} : j \le N) \le dim(a_{1j}a_{2j} : j \le j_0)+(N-j_0),
\end{displaymath}
so $dim(a_{1j}a_{2j} : j \le j_0) \ge 2+j_0$.
Hence, as $dim(a_{1j_0}a_{2j_0} /a_{1,<j_0}a_{2,<j_0})=2$, we have $dim(a_{1j}a_{2j} : j < j_0) \ge j_0$.
From this and the fact that the claim holds below $j_0$ it follows that $dim(a_{11}a_{21})=2$ and that for some $j'<j_0$, $dim(a_{11}a_{21}a_{1j'}a_{2j'})=3$.
Then, $dim(a_{11}a_{21} a_{1j'}a_{2j'} a_{1j_0}a_{2j_0})=5$, and as above, we contradict Claim \ref{permutod}.
\end{proof}
 
We now show that any $m \times n$ rectangle from $a$ has dimension at least $2m+n-2$ which will prove the lemma.
Suppose there are $m$ and $n$ so that for some $m \times n$ rectangle $C$ from $a$, $dim(C)<2m+n-2$.
Using claims \ref{indices1} and \ref{indices2}, we see that the inequality remains strict for any rectangle from $a$ that contains $C$.
Thus, $dim(a; M,N)<2M+N-2$.
But this contradicts the assumption that $dim(A; M,N)=2M+N-1$ and the specialization $A \to a$ has rank $1$.
\end{proof}

\subsection{Families of plane curves}\label{families}
 
In the proof of our main theorem, we will find a field configuration in the form of an uncountable indiscernible array of type $2m+n-2$, and the existence of an algebraically closed field will follow from Lemma \ref{groupexists}.
When constructing the array, families of plane curves will play a crucial role.
We give a definition analoguous to the one in \cite{HrZi}.
   
\begin{definition}\label{goodfor}
Let $C \subset \M^{n+m}$ be an irreducible set.
We say an element $a \in \M^n$ is \emph{good} for $C$ if there is some $b \in \M^m$ so that $(a,b)$ is a generic element of $C$.
\end{definition}

\begin{definition}
Let $\M$ be a Zariski-like structure, and let $E \subseteq \M^n$ and $C \subseteq \M^2 \times E$ be irreducible sets.
For each $e \in E$, denote $C(e)=\{(x,y) \in \M^2 \, | \, (x,y,e) \in C\}$.
Suppose now $e \in E$ is a generic point.
If $e$ is good for $C$ and the generic point of $C(e)$ has dimension $1$ over $e$, then we say that $C(e)$ is a \emph{plane curve}.
We say $C$ is a \emph{family of plane curves} parametrized by $E$.  

We say that $\alpha$ is the \emph{canonical parameter} of the plane curve $C(e)$ if $\alpha=\textrm{Cb}(x,y/e)$ for a generic element $(x,y) \in C(e)$.
We define the $\emph{rank}$ of the family to be the dimension of $\textrm{Cb}(x,y/e)$ over $\emptyset$, where $e \in E$ is generic, and $(x,y)$ is a generic point of $C(e)$.
\end{definition}
 
\begin{remark}\label{relevant}
We note that if $C \subset \M^2 \times E$ is a family of plane curves of rank at least $2$, then for a generic $e \in E$ and a generic point $(x,y) \in C(e)$, it holds that $x \notin \textrm{bcl}(e)$ and $y \notin \textrm{bcl}(e)$.
Indeed, if for instance $x \in \textrm{bcl}(e)$, then $y \da ex$ so we can choose $Cb(x,y/e)=Cb(x/e)$, and this canonical parameter has rank at most $1$.
\end{remark}

Our main theorem concerns Zariski-like structures with a non locally modular canonical pregeometry.
The non local modularity comes to play in finding a family of plane curves that can be parametrized with a tuple of dimension $2$.
This family will eventually be used to construct a suitable indiscernible array, following the example of \cite{HrZi}.

The following lemma states the existence of a suitable family and that there is an analogue to
the results of the computations in the beginning of the proof of Lemma 6.10 \cite{HrZi}.

\begin{lemma}\label{loytyyperhe} 
Let $\M$ be a quasiminimal pregeometry structure such that bcl is non locally modular.
Then, there exists a family of plane curves $C$, parametrized by a set $E$, and a tuple $d \in \M$ such that for a generic $e \in E(d)$, it holds that $dim(e/d)=2$ and $e$ is interbounded over $d$ with the canonical parameter $\alpha=Cb(x,y/d,e)$ of the family.

Moreover, if $(x,y)$ and $(x',y')$ are generic points of $C(d,e)$ such that $t(xy/de)=t(x'y'/de)$ and $xy \da_{de} x'y'$, then $e \in \textrm{bcl}(d,x,y,x',y')$ and $dim(xyx'y'/d)=4$.
\end{lemma}
 
\begin{proof}
Since $\M$ is not locally modular, there exists, by Lemma 4.13 in \cite{lisuriart},   
a  family of plane curves that has rank $r \ge 2$.
Let $\alpha$ be the canonical parameter of a generic curve in this family.
Then, $dim(\alpha/\emptyset)=r$.
Let $(x,y)$ be a generic point on this curve, so $\alpha=Cb(x,y/\alpha)$.
  
Let $x_1, \ldots, x_r, y_1, \ldots, y_r$ be such that $t(x_i,y_i/\alpha)=t(x,y/\alpha)$ for each $i$ and the sequence $x, x_1, \ldots, x_r$ is independent over $\alpha$.
Denote  $d=(x_1, \ldots, x_r,y_1, \ldots, y_{r-2})$ and $e=(y_{r-1}, y_r)$.
We will show that $\alpha=Cb(x,y/d,e)$, that $dim(d,e)=r$, and that $\alpha$ is interbounded with $e$ over $d$. 
Then, taking $C$ to be the locus of $(x,y,d,e)$ and $E$ to be the locus of $(d,e)$, we get a family of plane curves $C$ parametrized by $E$ that is as wanted.

We show first that $\alpha \in \textrm{bcl}(d,e).$
Since $xy \da_\alpha de$, it will then follow that $\alpha=Cb(x,y/de)$.
For $k \le r$, we have $\alpha=Cb(x_k,y_k/\alpha)$.
Hence, if  
\begin{eqnarray}\label{cbalpha}
\alpha \da_{x_1, y_1, \ldots, x_{k-1}, y_{k-1}} x_k y_k,
\end{eqnarray}
it follows by symmetry and the properties of canonical bases that 
$$\alpha=Cb(x_k,y_k/\alpha, x_1, y_1, \ldots, x_{k-1}, y_{k-1}),$$
and thus, applying again symmetry to (\ref{cbalpha}), that $\alpha \in \textrm{bcl}(x_1, y_1, \ldots, x_{k-1}, y_{k-1})$.
Since the dimension of $\alpha$ can drop at most $r$ times, $\alpha \in \textrm{bcl}(d,e)$.

The sequence $(x_i, y_i)_{1 \le i \le r}$ is independent over $\alpha$, and for each $i$, we have $dim(x_i, y_i/\alpha)=1$.
We get
$$dim(x_1, y_1, \ldots, x_r, y_r)=dim(x_1, y_1, \ldots, x_r,y_r, \alpha)=2r.$$
We also see that 
$dim(\alpha/x_1, y_1, \ldots, x_k, y_k)=r-k$, and thus $\alpha$ is interbounded with $e$ over $d$.

The claim after ``moreover" is proved with essentially the same calculations that are presented above.
\end{proof}

\subsection{Coding elements of $\M^{eq}$}\label{cooding} 
 
By Theorem 4.19 in \cite{lisuriart} (4.22 in \cite{lisuri}), there is a group interpretable in a Zariski-like structure with non-trivial pregeometry.
We will use the elements of this group to construct the indiscernible arrays needed to prove our main theorem.
However, the group will be in $\mathbb{M}^{eq}$, and we have to construct the arrays in $\mathbb{M}$.
Thus, we present a way to code the elements of $\mathbb{M}^{eq}$. 

\begin{definition}
Let $\alpha \in \M^{eq}$ and $c \in \M$.
We say that a tuple $a \in \M$ is a \emph{weak code} for $\alpha$ over $c$ if $\alpha \in \textrm{dcl}(c, a)$ and
$a$ is interbounded with $\alpha$ over $c$.
\end{definition}
    
\begin{lemma}\label{koodaus}
Let $S$ be a set whose elements have the same Galois type, and suppose that type is stationary. 
Then, there is a tuple $c \in \M$ such that the elements of $S$ have distinct weak codes over $c$.
Moreover, there is a definable function $F$ such that if $\alpha \in S$ and $a$ is a weak code for $\alpha$ over
$c$ in this coding, then $\alpha=F(c,a)$.   
\end{lemma}

\begin{proof}
Let $\alpha \in S$.
Then, there exists a tuple $b \in \M$ and some definable function $F$ such that $\alpha=F(b)$.
Here, $F$ is a composition of function symbols of the form $F_E$ for some equivalence relations $E$ (remember that we use the notation $\M^{eq}$ for $(\M^{eq})^{eq}$, etc).
Suppose now that $b$ is chosen so that $dim(b/\alpha)$ is the least possible, and denote this dimension by $n$.
We may assume that $b=(b_1, \ldots, b_n, b_{n+1}, \ldots, b_m)$, where $dim(b_1, \ldots, b_n/\alpha)=n$,
and hence 
\begin{eqnarray}\label{bbcl}
b \in \textrm{bcl}(b_1, \ldots, b_n, \alpha).
\end{eqnarray} 
Choose now some tuple $c=(c_1, \ldots, c_n) \in \M$ such that $c \da_\emptyset S$ and $Lt(c/\emptyset)=Lt(b_1, \ldots, b_n/\emptyset)$.
By stationarity of Lascar types, $Lt(c_1, \ldots, c_n/\alpha)=Lt(b_1, \ldots, b_n/\alpha)$,
so there is some automorphism $f$ such that $f(\alpha)=\alpha$ and $f(b_1, \ldots, b_n)=c$.
Write $a=(f(b_{n+1}),...,f(b_m))$.
Then $a$ is a weak code for $\alpha$ over $c$.
Indeed, $\alpha=F(f(b))=F(c,a)$, so $\alpha \in \textrm{dcl}(c,a) \subseteq \textrm{bcl}(c,a)$, and from (\ref{bbcl}) it follows that $a \in \textrm{bcl}(\alpha,c)$.
 
If $\alpha' \in S$, then since $\alpha$ and $\alpha'$ realize the same stationary type,  
$t^g(\alpha/c)=t^g(\alpha'/c)$.
Thus, there is some $b'=(b_1', \ldots, b_n', b_{n+1}', \ldots, b_{m}') \in \M$ such that $\alpha'=F(b')$ and $t^g(b'/c)=t^g(b/c)$, and hence $Lt(b_1', \ldots, b_n'/\emptyset)=Lt(c/\emptyset)$. 
As above, we obtain a weak code $a'$ for $\alpha'$ over $c$.
If $a'=a$, then also $\alpha=F(c,a)=F(c,a')=\alpha'$, so distinct elements of $S$ have distinct weak codes.
\end{proof} 
 
\begin{remark}
Note that the coding provided by Lemma \ref{koodaus} is not unique.
Indeed, for each $\alpha \in S$, there may be several distinct elements $b$ such that $\alpha=F(b)$, and each of them gives a distinct weak code $a$. 
\end{remark}

\begin{remark}\label{koodaus2}
Since all the elements of $S$ have the same Galois type and that type is stationary, 
they also have the same Lascar type. 
Thus, we may choose the codes so that they have the same Lascar type.
Moreover, if $d$ is a tuple independent from $S$, then we can choose $c$ so that it is independent from $d$ (in the proof, just choose  $c$ so that $c \da Sd$). 
\end{remark}

The following technical lemma will be used in the proof of the main theorem.
It is an analogue for 6.6(6) in \cite{HrZi}.

\begin{lemma}\label{binterbound}
Let $(G,+)$ be a $1$-dimensional Galois definable group in $\M^{eq}$.
Let  $\alpha_i, \beta_i, \alpha_i'$, $i=1,2$, and $\beta^*$ be elements of $G$, generic over $c$, coded by the weak codes $a_i, b_i, a_i',b^*$, respectively (over $c$). 
Suppose $\alpha_1-\alpha_2=\beta_1-\beta_2$.
Let $\gamma_1, \gamma_2 \in G$ be such that $\gamma_1 \gamma_2 \da c a_1 a_2 b_1 b_2 b^*$ and $\gamma_1 \gamma_2 \da c a_1'a_2' b^*$, and suppose $g_i$, $h_i$, $g_i'$, $h_i'$, $i=1,2$  are weak codes (over $c$) for $\alpha_i+\gamma_i$, $\beta_i+\gamma_i$, $\alpha_i'+\gamma_i$, and $\beta^*+\gamma_i$, respectively. 
Suppose $dim(a_1,a_2, b_1/c)=3$ and 
$$(c,a_1, a_2, b_1, b_2, g_1, g_2, h_1,h_2) \to (c,a_1',a_2' ,b^*, b^*,g_1',g_2',h_1', h_2').$$

Then, $a_1'$ and $a_2'$ are interbounded over $c$.
\end{lemma} 
 
\begin{proof}
Since we have assumed that there are no non-classical groups in our monster model, we have
\begin{eqnarray}\label{summa}
(\alpha_1+\gamma_1)-(\alpha_2+\gamma_2)=(\beta_1+\gamma_1)-(\beta_2+\gamma_2),
\end{eqnarray}
and that $dim(g_1, g_2, h_1, h_2/c)=3$.
We divide the proof into two cases, depending whether $dim(g_1',g_2',h_2',h_2'/c)=3$ or not.

Suppose first the equality holds.
Then, $t(g_1, g_2, h_1, h_2/c)=t(g_1',g_2',h_1',h_2'/c)$, so (\ref{summa}) implies
$$(\alpha_1'+\gamma_1)-(\alpha_2'+\gamma_2)=(\beta^*+\gamma_1)-(\beta^*+\gamma_2),$$
in particular $\alpha_1'=\alpha_2'$.
Since $a_1'$ and $a_2'$ are (possibly distinct) weak codes for $\alpha_1'$, they must be interbounded over $c$.

Suppose now $dim(g_1',g_2',h_1',h_2'/c)<3$.
We show that $\alpha_1'-\alpha_2' \in \textrm{bcl}(c)$ and the claim will follow.
Assume towards a contradiction that this is not the case.
Denote 
$$A=\{\alpha_1'+\gamma_1,\alpha_2'+\gamma_2,\beta^*+\gamma_1,\beta^*+\gamma_2\}.$$ 
Since it holds that
$$\alpha_1'-\alpha_2'=(\alpha_1'+\gamma_1)-(\alpha_2'+\gamma_2)-(\beta^*+\gamma_1)+(\beta^*+\gamma_2),$$
 we have $\alpha_1'-\alpha_2' \in \textrm{bcl}(A)$.
If $dim(\alpha_1'-\alpha_2', \beta^*+\gamma_1, \beta^*+\gamma_2/c)=3$, then $dim(A/c)=3$, which contradicts the assumption that $dim(g_1',g_2',h_1',h_2'/c)<3$.
Since $\beta^*+\gamma_1$ is independent from $\beta^*+\gamma_2$, we thus have 
$\alpha_1'-\alpha_2' \in \textrm{bcl}(\beta^*+\gamma_1, \beta^*+\gamma_2,c)$.

We have $\alpha_1'-\alpha_2' \notin \textrm{bcl}(c, \beta^*)$ since otherwise, applying exchange, we would get $\beta^* \in \textrm{bcl}(\beta^*+\gamma_1, \beta^*+\gamma_2,c)$ which is impossible since $dim(\beta^*, \gamma_1, \gamma_2/c)=3$.
But since $\alpha_1'-\alpha_2' \da_{\beta^* c} \gamma_1 \gamma_2$, this implies $\alpha_1'-\alpha_2' \notin \textrm{bcl}(c, \beta^*, \gamma_1, \gamma_2)$, a contradiction.
This proves the lemma.
\end{proof} 

\subsection{The main theorem}\label{prooof}

When proving our main theorem,
we follow the outline of the proof of Lemma 6.11 in \cite{HrZi}.
From now on, we will suppose that $\M$ is a non locally modular Zariski-like structure and that $\M$ does not interpret a non-classical group.
By Theorem 4.19 in \cite{lisuriart} (see also \cite{lisuri}, 4.22), there is a Galois definable, $1$-dimensional group $G$ in $\M^{eq}$.
This group plays a crucial role when proving that $\M$ interprets a field, and we will eventually use Lemma \ref{koodaus} to (weakly) code some of the generic elements in $G$.
When doing so, we will always suppose that if $\alpha \in G$ is generic, $a=(a_1, \ldots, a_m)$ is a weak code for $\alpha$, and $c$ is the parameter tuple from Lemma \ref{koodaus}, then $dim(a_1/c)=1$.
 
Let $C$ be a family of plane curves, parametrized by $E$, and let $d \in \M$ be such that for generic $e \in E$, it holds that $dim(e/d)=2$ (such a family exists by Lemma \ref{loytyyperhe}).
For $x,y,x',y' \in \M$ and $e$ a generic element in $E(d)$, we write $C^2(e; xy, x'y')$ if the following hold:
\begin{enumerate}[1.]
\item $(x,y)$ and $(x',y')$ are generic on $C(d,e)$;
\item $xy \da_{de} x'y'$;
\item $Lt(xy/de)=Lt(x'y'/de).$
\end{enumerate}

\begin{lemma}\label{loytyye}
Let $A \subseteq \M^{eq}$ be a $1$-dimensional Galois definable set with unique generic type, and let $\alpha, \beta, \alpha',\beta'$ be generic elements of $A$ independent from each other and from $d$.
Let $a=(a_1, \ldots, a_m)$, $b=(b_1, \ldots, b_m)$, $a'=(a_1, \ldots, a_m)$ and $b'=(b_1, \ldots, b_m)$ be weak codes for $\alpha$, $\beta$, $\alpha'$ and $\beta'$, respectively, over the parameter tuple $c$, chosen so that $dim(a_1/c)=dim(b_1/c)=dim(a_1'/c)=dim(b_1'/c)=1$. 

Suppose $h,g,g' \in \M$ are such that $h \da_c aba'b'd$, $g \in \textrm{bcl}(h,b)$, and $Lt(a,b,g/cdh)=Lt(a'b',g'/cdh)$.
Then, there exists some $e \in E(d)$ such that $C^2(e; a_1b_1, a_1' b_1')$
and $Lt(a,b,g/cdeh)=Lt(a',b',g'/cdeh).$
\end{lemma}

\begin{proof}
Let $e' \in E(d)$ be generic, and let $x,y,x',y' \in \M$ be such that $C^2(e';xy,x'y')$ and $xyx'y' \da_d ch$.
By Lemma \ref{loytyyperhe}, the sequence $x,y,x',y'$ is generic and independent over $d$, and hence over $cdh$.
It follows from the assumptions that also $a_1,b_1,a_1',b_1'$ is generic and independent over $cdh$.
By (QM4) (the uniqueness of generic type in a quasiminimal pregeometry structure), there is some $f \in \textrm{Aut}(\M/cdh)$ such that $f(x,y)=(a_1, b_1)$.
Let $e''=f(e')$, $x''=f(x')$, and $y''=f(y')$.
Then, $C^2(e'';a_1b_1,x''y'')$, and thus $Lt(a_1,b_1/de'')=Lt(x'',y''/de'').$
By Lemma \ref{loytyyperhe}, $e'' \in \textrm{bcl}(d, a_1, b_1, x'',y'')$, so $a_1b_1 x''y'' \da_{de''} ch$.
It follows that $Lt(a_1,b_1/cde''h)=Lt(x'',y''/cde''h).$
Now there are some elements $x_2, \ldots, x_m$, $y_2, \ldots, y_m$, and $g^*$ such that
$$Lt(a_1, \ldots, a_m, b_1, \ldots, b_m,g/cde''h)=Lt(x'',x_2, \ldots, x_m, y'', y_2, \ldots, y_m,g^*/cde''h).$$
Denote $x^*=(x'',x_2, \ldots, x_m)$ and $y^*=(y'', y_2, \ldots, y_m).$
By our assumptions, $Lt(a'b'g'/cdh)=Lt(abg/cdh)=Lt(x^*y^*g^*/cdh).$
Since $a'b'g' \da_{cdh} ab$ and $x^*y^*g^* \da_{cdh} ab$, there is some $f' \in \textrm{Aut}(\M/abcdh)$ such that $f'(x^*y^*g^*)=a'b'g'$.
Let $e=f'(e'')$.
Then, $e$ is as wanted.
\end{proof}

\begin{lemma}\label{loytyyspec}
Let $A \subseteq \M^{eq}$ be a $1$-dimensional Galois definable set with unique generic type, and let $\alpha, \beta, \alpha',\beta'$ be generic elements of $A$.
Let $a=(a_1, \ldots, a_m)$, $b=(b_1, \ldots, b_m)$, $a'=(a_1, \ldots, a_m)$ and $b'=(b_1, \ldots, b_m)$ be weak codes for $\alpha$, $\beta$, $\alpha'$ and $\beta'$, respectively, over the parameter tuple $c$, chosen so that
$dim(a_1/c)=dim(b_1/c)=dim(a_1'/c)=dim(b_1'/c)=1$. 
Let $e \in E(d)$ be generic and suppose $C^2(e; a_1b_1, a_1'b_1')$.
Suppose $h,g,g' \in \M$ are such that $h \da_c aba'b'd$, $g \in \textrm{bcl}(h,b)$, and suppose $Lt(abg/cdeh)=Lt(a'b'g'/cdeh).$

Then, there is a specialization $(c,d,e,h,a,b,a',b',g') \to (c,d,e,h,a,b,a,b,g).$
\end{lemma}
 
\begin{proof} 
We will apply Axiom (8) of Zariski-like structures.
By Lemma \ref{loytyyperhe} and our assumptions, $dim(a_1, b_1, a_1', b_1'/cdh)=4$, and $e \in \textrm{bcl}(d,a_1, b_1, a_1', b_1').$
Thus, \mbox{$a_1 b_1 \da_{cdeh} a_1' b_1'$,} and it follows that $abg \da_{cdeh} a'b'g'$.
By Lemma \ref{si}, $abg$ and $a'b'g'$ are strongly indiscernible over $cdeh$, and of course also $abg$ and $abg$ are. 
Since $abg$ and $a'b'g'$ have the same Lascar type over $cdeh$, we have $(c,d,e,h,a,b,g) \to (c,d,e,h,a',b',g').$
Moreover, $cdeh \to cdeh$ is strongly good and of rank $0 \le 1$.
 Hence, the Lemma follows from Axiom (8).
\end{proof}

We will now prove our main theorem, 
the analogue to Lemma 6.11 in \cite{HrZi}.

\begin{theorem}\label{MAIN}
Let $\M$ be Zariski-like structure with a non locally modular pregeometry.
Then, $\M$ interprets either an algebraically closed field or a non-classical group.
\end{theorem}

\begin{proof}
Assume $\M$ doesn't interpret a non-classical group.
We will prove the theorem by finding a rank-indiscernible array of type $2m+n-2$ over a certain finite set of parameters and applying Lemma \ref{groupexists}.
By Theorem 4.19 in \cite{lisuriart} (or 4.22 in \cite{lisuri}), there is a Galois definable $1$-dimensional group $(G,+)$ in $\M^{eq}$ with unique generic type, and by \ref{1abelian}, $G$ is Abelian. 
The elements of $G$ will be used when constructing the array.
 
By Lemma \ref{loytyyperhe}, there exists a family of plane curves $C$, parametrized by a set $E$, and a tuple $d \in \M$ such that for generic $e \in E(d)$, it holds that $dim(e/d)=2$ and $e$ is interbounded over $d$ with the canonical parameter of the family.
Let $\kappa$ be some cardinal large enough (for the argument that follows after a couple of paragraphs), and let $\alpha_i$, $\beta_i$, $\beta^j$, $i,j<\kappa$, be generic  
elements of $G$ independent from each other and from $d$.
For $i,j \ge 1$, let $\beta_{ij}=\beta_i+\beta^j$.
Let $\gamma_1, \ldots, \gamma_8$ be elements of $G$ generic and independent over everything mentioned so far. 
Denote $\gamma=(\gamma_1, \ldots, \gamma_8)$ and $\gamma_{ij}=(\beta_{ij}+\gamma_1, \ldots, \beta_{ij}+\gamma_8)$.
The elements $\alpha_i$, $\beta_{ij}$ will be used to build the indiscernible array, the others are auxiliary elements that will be needed in some of the calculations.

Denote $S=\{\alpha_i, \beta_i, \beta^j, \beta_{ij}, \gamma_k, \beta_{ij}+\gamma_k\}_{i,j<\kappa, 1 \le k \le 8}$.
By Lemma \ref{koodaus} and Remark \ref{koodaus2}, there is some $c \in \M$, independent from $Sd$, so that the elements of $S$ can be (weakly) coded over $c$. 
In particular, there are weak codes in $\M$ for the elements of $S$
so that each group element is interbounded over $c$ with its weak code, and the first coordinates of the weak codes are generic over $c$.
Fix now some such set of weak codes, and let $a_i$, $b_i$, $b^j$ and $b_{ij}$ denote the weak codes for $\alpha_i$, $\beta_i$, $\beta^j$ and $\beta_{ij}$, respectively.
Let $g$ be a tuple consisting of the weak codes for $\gamma_k$, $1 \le k \le 8$, and let $g_{ij}$ be a tuple consisting of the weak codes for $\beta_{ij}+\gamma_k$, $1 \le k \le 8$.
 
\begin{claim}\label{lisatty}
We may choose the weak codes so that  for all $i,j$,
$$Lt(a_i, b_{ij}, g_{ij}/cdg)=Lt(a_0,b_{00}, g_{00}/cdg).$$
\end{claim}

\begin{proof}
Since $S \da cd$, we may code the elements of $S$ so that all the codes have the same Lascar type over $cd$.
Let $i,j<\kappa$.
The sequence $a_0, b_{00}, a_i, b_{ij}, g$ is independent, so we have 
$Lt(a_i,b_{ij}/cdg)=Lt(a_0b_{00}/cdg)$.
By Lemma 2.53 in \cite{lisuriart}, there is an automorphism $f \in Aut(\M/cgd)$ such that
$f(a_0b_{00})=a_ib_{ij}$ and $Lt(a_i,b_{ij},f(g_{00})/cdg)=Lt(a_0,b_{00}, g_{00}/cdg)$.
We claim that $f(g_{00})$ is a weak code for $\gamma_{ij}$ over $c$, and thus we may choose $g_{ij}=f(g_{00})$.
Indeed, by Lemma \ref{koodaus}, there is a definable function $F$ such that $F(c,g)=\gamma$ (abusing notation to mean that we apply the function
$F$ to each member of the tuple $g$ separately), 
and thus $f(\gamma)=\gamma$.
Moreover, $F(c,b_{00})=\beta_{00}$ and $F(c,b_{ij})=\beta_{ij}$, so $f(\beta_{00})=\beta_{ij}$, and thus
$$F(c,f(g_{00}))=f(\gamma_{00})=f(\gamma + \beta_{00})=\gamma+\beta_{ij}=\gamma_{ij},$$
as wanted. 
\end{proof}

Suppose from now on that the codes are chosen so that $Lt(a_i, b_{ij}, g_{ij}/cdg)=Lt(a_0,b_{00}, g_{00}/cdg).$
Denote by $a_i^1$ and $b_{ij}^1$ the first coordinates of the tuples $a_i$ and $b_{ij}$, respectively.
By Lemma \ref{loytyye}, there is, for each pair $(i,j)$, some $e_{ij} \in E(d)$ such that $C^2(e_{ij}; a_0^1b_{00}^1, a_i^1b_{ij}^1)$ (in the notation introduced right before Lemma \ref{loytyye}) and $$Lt(a_i, b_{ij}, g_{ij}/cde_{ij}g)=Lt(a_0,b_{00}, g_{00}/cde_{ij}g).$$
Let $A_{ij}=(a_i, b_{ij}, e_{ij}, g_{ij})$, and let $A=(A_{ij})_{i,j\ge 1}$.
We will use $A$ to build an indiscernible array of type $2m+n-2$.  
Denote from now on $p=c d g a_0 b_{00} g_{00}$.

We will next show that if we choose $\kappa$ to be large enough, then we can find an indiscernible array of size $\o_1 \times \o_1$ such that each one of its finite subarrays is isomorphic to some finite subarray of $A$.
Let  $\lambda< \kappa$ be a cardinal large enough (but not too large) for the argument that follows. 
For each $i< \kappa$, denote $A_{i, <\lambda}=(A_{ij} | j< \lambda)$.
Using Erd\"os-Rado and an Ehrenfeucht-Mostowski construction, one finds a  sequence $(A'_{i, < \lambda})_{i<\o_1}$ such that every finite permutation of the sequence preserving the order of the indices $i$ extends to 
some $f \in \textrm{Aut}(\M/p)$.
Moreover, an isomorphic copy of every finite subsequence can be found in the original sequence $(A_{i, < \lambda})_{i<\kappa}$.
This construction is due to Shelah, and the details can be found in e.g. \cite{CatTran}, Proposition 2.13.
There it is done for a sequence of finite tuples (whereas we have a sequence of sequences of length $\lambda$), but the proof is essentially the same in our case.
 
We may now without loss  assume that $(A_{i,< \lambda}')_{i<\o_1}$ are the $\o_1$ first elements in the sequence 
$(A_{i,< \lambda})_{i<\kappa}$.
Since we have chosen $\lambda$ to be large enough, we may apply the same argument to $(A_{<\o_1, j}')_{j<\lambda}$ to obtain an array $(A_{<\o_1, j}'')_{j<\o_1}$. 
This is an array of size $\o_1 \times \o_1$, indiscernible over $p$, and we may assume it is a subarray of the original array $A$.
 From now on, we will use $A$ to denote $A''$.
 
Eventually, we will apply Lemma \ref{tekn2} to $A$ to obtain an array that is as wanted.
Thus, we next prove that the assumptions of the Lemma hold for every subarray of $A$ over the parameters $p$.
 
\begin{claim}\label{asstekn2}
$dim(A;m,n/p)=2m+n-1$ and 
$$dim(dcl(A_{11}A_{12}A_{13}p) \cap dcl(A_{21}A_{22}A_{23}p)/p)=2.$$
\end{claim}

\begin{proof} 
Denote $A'=(A_{ij})_{1 \le i \le m, 1 \le j \le n}$ and  
$P=\{a_i, b_{i1}, b_{1j}\}_{1 \le i \le m, 1 \le j \le n}.$
The set $P$ is independent over $p$.
By Lemma \ref{loytyyperhe}, we have $e_{ij} \in \textrm{bcl}(a_i, b_{ij}, p)$, and since $\beta_{ij}=\beta_{i1}+\beta_{1j}-\beta_{11}$, we have $b_{ij} \in \textrm{bcl}(b_{i1}, b_{1j}, b_{11},p)$.
Moreover, $g_{ij} \in \textrm{bcl}(p,b_{ij})$.
Thus, $A' \subseteq \textrm{bcl}(P)$, so $dim(A'/p)=\vert P \vert = 2m+n-1$.
Since $A$ is indiscernible, $dim(A;m,n/p)=2m+n-1$.

For the rest of the claim, denote $C=A_{11}A_{12}A_{13}$ and $C'=A_{21}A_{22}A_{23}$.
The group elements $\beta^1-\beta^2$ and $\beta^1-\beta^3$ are independent over $p$.
Since $\beta^1-\beta^j=\beta_{11}-\beta_{1j}=\beta_{21}-\beta_{2j}$, they are in $\textrm{dcl}(Cp) \cap \textrm{dcl}(C'p)$ and this set has dimension at least $2$.
If it would be greater than $2$, then, since $dim(C/p)=4$, we would have $dim(C/p(dcl(Cp) \cap dcl(C'p))) \le 1$, so
\begin{eqnarray*}
dim(C \cup C'/p)&=&dim(C'/p)+dim(C/C'p) \\ &\le& dim(C'/p)+dim(C/p(dcl(Cp) \cap dcl(C'p))) \le  5.
\end{eqnarray*}
But by the type of the array $A$, we have $dim(C \cup C'/p)=6$, a contradiction.
\end{proof}

Denote now $A_{ij}'=A_{1j}$, and let $A'=(A_{ij}')_{0<i,j<\o_1}$.  

Write $x \to_p y$ for $px \to py$.

\begin{claim}\label{spec1}
$A \to_p A'.$
\end{claim}
 
 \begin{proof}
It suffices to prove that for every finite $J \subset \o_1$, the claim holds for all $\o_1 \times J$ -subarrays of $A$ and $A'$.
Denote $P=\{b^j \, | \, j \in J\}$.
The set $\{a_i, b_i \, | \, i \in I\}$ is independent over $pP$.
By Lemma \ref{loytyyperhe} and our coding, $e_{ij}, b_{ij}, g_{ij} \in \textrm{bcl}(P,p, a_i,b_i)$ for every $j \in J$,
so the sequence $(a_i, b_{iJ}, e_{iJ}, g_{iJ})_{i <\o_1}$ is independent over $pP$ (here, $b_{iJ}=(b_{ij})_{j \in J}$, etc).
By Lemma \ref{morleycofinal}, there is some $X \subset \o_1$, cofinal in $\o_1$ such that $(a_i, b_{iJ}, e_{iJ})_{i \in X}$ is strongly indiscernible over $pP$.
Denote by $Q$ the sequence (obtained from $A'$) that repeats $\o_1$ many times the entry $(a_1, b_{1J}, e_{1J}, g_{1J})$.
Clearly, it is strongly indiscernible over $pP$. 
By Axiom (8) of Zariski like structures,
$$pP(a_i, b_{iJ}, e_{iJ}, g_{iJ})_{i \in X} \to pPQ.$$
Since the array $A$ is indiscernible over $p$, we have (after relabelling the indices in $X$ with $\o_1$)
$$p(a_i, b_{iJ}, e_{iJ}, g_{iJ})_{i \in \o_1} \to p(a_i, b_{iJ}, e_{iJ}, g_{iJ})_{i \in X},$$
and the claim follows.
\end{proof}

Denote $A_{ij}''=(a_{00}, b_{00}, e_{1j}, g_{00})$, and let $A''=(A_{ij}'')_{0<i,j<\o_1}$.  

\begin{claim}\label{spec2}
$A' \to_p A''$.
\end{claim}
 
\begin{proof}
Since specializations respect repeated entries, it suffices to show that
$$(a_1, b_{1j}, e_{1j}, g_{1j})_{j<\o_1} \to_p (a_0, b_{00}, e_{1j}, g_{00})_{j<\o_1}.$$
By Lemma \ref{loytyyspec}, this holds for every $j$ individually.
Here, both sequences are independent over $p$, and
by the indiscernibility of the array $A$, Lemma \ref{samelascar}
and Lemma \ref{si}, they are strongly indiscernible.
Since $a_1$ is independent from $p$, the specialization $a_1 \to_p a_0$ is strongly good, 
and the claim follows from Axiom (8) of Zariski-like structures.
\end{proof}
 
By claims \ref{spec1} and \ref{spec2}, $A \to_p A''$.
We will apply Axiom (9) of Zariski-like structures to the latter specialization and eventually obtain an infinite rank-indiscernible array $A^*$ such that $A \to_p A^* \to_p A''$, rank indiscernible of type  $2m+n-2$ over the parameters $p$.

Let $A^0$ be a finite subarray of $A$ containing the entry $A_{11}$, and let ${A^0}''$ be the corresponding subarray of $A''$. 
By Lemma \ref{finarray}, there is an array ${A^0}^*$ and $b_{11}^*, e_{11}^*, g_{11}^* \in \M$ such that  $A^0 \to_p  {A^0}^* \to_p {A^0}''$, the first specialization is of rank $1$, and  ${A^0}^*_{11}=(a_0, b_{11}^*, e_{11}^*, g_{11}^*)$.
Write ${A^0}^*_{ij}=(a_i^*,b_{ij}^*, e_{ij}^*, g_{ij}^*)$.
We will show that the assumptions posed for $a$ in the statement of Lemma \ref{tekn2} hold for ${A^0}^*$ over $p$.
It will then follow that ${A^0}^*$ is rank indiscernible of type $2m+n-2$ over $p$.
On the way, we will see that $a_i^*=a_0$ and $b_{ij}^*=b_{00}$ for all $i,j$.  
We do the proof as a series of claims.

\begin{claim}\label{eij}
Let $J$ be the set of $j$-indices corresponding to a row of ${A^0}^*$.
Then, for each $i$, the sequence $e_{iJ}^*$ is independent over $p$, and $dim(e_{ij}^*/p)=1$ for all $i,j$.
\end{claim}

\begin{proof}
We have $e_{iJ} \to_p  e_{iJ}^* \to_p e_{1J}$.
Since $A$ is indiscernible over $p$, we have $t(e_{iJ}/p)=t(e_{1J}/p)$, so the specializations are isomorphisms.
Thus, the sequence is independent over $p$ since the sequence $e_{iJ}$ is.
Moreover, it follows from Lemma \ref{loytyyperhe} and the choice of the tuples that form the sequence $p=cdga_0b_{00}g_{00}$ that $dim(e_{ij}/p)=1$ and thus $dim(e_{ij}^*/p)=1$.
\end{proof}

\begin{claim}\label{b00}
For each $j$, we have $b_{1j}^*=b_{00}$.
\end{claim}

\begin{proof}
We have $cda_1b_{1j}e_{1j} \to cda_0b_{1j}^*e_{1j}^*\to cd a_0 b_{00} e_{1j}$ (remember that $a_1^*=a_0$), and $dim(a_1b_{1j}e_{1j}/cd)=dim(a_0 b_{00} e_{1j}/cd)$. Thus, $t(a_0b_{1j}^*e_{1j}^*/cd)=t(a_0 b_{00} e_{1j}/cd)$, and in particular $b_{1j}^* \in \textrm{bcl}(c,d, a_0, e_{1j}^*)$.
Applying this and Claim \ref{eij}, we get $dim(a_0b_{1j}^*e_{1j}^*/cda_0b_{00})=1$, and since also $dim(a_0 b_{00} e_{1j}/cda_0b_{00})=1$, the specialization $$cda_0b_{00}a_0b_{1j}^*e_{1j}^* \to cda_0b_{00}a_0b_{00}e_{1j}$$ is an isomorphism. 
This implies $b_{1j}^*=b_{00}$ by Remark \ref{isoinj}.
\end{proof}

For $k=1, \ldots, 8$, denote the weak code for $\gamma_k$ by $g_k$, and the weak code for $\beta_{ij}+\gamma_k$ by $f_k$ (then, $g=(g_1, \ldots, g_8)$ and $g_{ij}=(f_1, \ldots, f_8)$).
Similarly, write $g_{ij}^*=(f_1^*, \ldots, f_8^*)$ and $g_{00}=(f_1^0, \ldots, f_8^0)$.

\begin{claim}\label{kk'}
For any indices $i,j,j'$, there are some $1 \le k<k' \le 8$ so that $g_kg_{k'} \da_c b_{00}b_{ij}^*b_{ij'}^*$.
\end{claim}

\begin{proof}
Suppose that $g_kg_{k'} \nda_c b_{00}b_{ij}^*b_{ij'}^*$ for all $1 \le k<k' \le 8$.
Since the sequence $g_1, \ldots, g_8$ is independent over $c$, we then have
\begin{eqnarray*}
dim(b_{00}b_{ij}^*b_{ij'}^*/c)&>&dim(b_{00}b_{ij}^*b_{ij'}^*/cg_1g_2)>dim(b_{00}b_{ij}^*b_{ij'}^*/cg_1g_2g_3 g_4)\\&>&dim(b_{00}b_{ij}^*b_{ij'}^*/cg_1\cdots g_6)>dim(b_{00}b_{ij}^*b_{ij'}^*/cg_1\cdots g_8),
\end{eqnarray*}
which is impossible since $dim(b_{00}b_{ij}^*b_{ij'}^*/c) \le 3$.
\end{proof}

\begin{claim}\label{binterbounded2}
For any indices $i,j,j'$, the elements $b_{ij}^*$ and $b_{ij'}^*$ are interbounded over $c$.
\end{claim}

\begin{proof} 
Relabelling the indices if necessary, we may by Claim \ref{kk'} assume that $g_1g_2 \da_c b_{00}b_{ij}^*b_{ij'}^*$.
The specialization $A \to_p A^* \to_p A''$ gives 
$$b_{ij} f_k c g_k \to b_{ij}^*f_k^*cg_k \to b_{00}f_k^0cg_k$$
and by our assumptions $Lt(b_{ij}f_k /cg_k)=Lt(b_{00}f_k^0/cg_k)$,
so it follows from Axiom (3) of Zariski-like structures that
 $t(b_{ij}f_k/cg_k)=t(b_{ij}^*f_k^*/cg_k)$, so $b_{ij}^*$ can be interpreted as a weak code (over $c$) for some $\beta_{ij}^* \in G$ generic over $c$ and $f_k^*$ as a weak code (over $c$) for $\beta_{ij}^*+\gamma_k$.

Since $\beta_{ij}-\beta_{ij'}=\beta_{1j}-\beta_{1j'}$ (remember that $\beta_{ij}=\beta_i+\beta^j$), the statement follows by Claim \ref{b00} and Lemma \ref{binterbound}.
\end{proof}

\begin{claim}\label{a0b00}
For all $i,j$, $a_i^*=a_0$, and $b_{ij}^*=b_{00}$.
\end{claim}

\begin{proof}
Suppose $j \neq j'$.
We will show that the specialization 
$$cda_0b_{00}a_i^*b_{ij}^*e_{ij}^*e_{ij'}^* \to c da_0b_{00}a_0b_{00} e_{1j} e_{1j'}$$
is an isomorphism, and the claim will follow by Remark \ref{isoinj}.
Since $dim(a_0b_{00} e_{1j} e_{1j'}/cda_0b_{00})=2$, it will suffice to show that also $dim(a_i^*b_{ij}^*e_{ij}^*e_{ij'}^*/cda_0b_{00})=2$.

As before, denote by $a_0^1$ and $b_{00}^1$ the first coordinates of $a_0$ and $b_0$, respectively, and by $(a_i^*)^1$ and $(b_{ij}^*)^1$ those of $a_i^*$ and $b_{ij}^*$, respectively. 
We have $d a_i b_{ij} e_{ij} \to d a_i^* b_{ij}^* e_{ij}^* \to d a_0 b_{00} e_{1j}$, so $((a_i^*)^1, (b_{ij}^*)^1, e_{ij}^*)$ 
(and similarly, $((a_i^*)^1, (b_{ij'}^*)^1, e_{ij'}^*)$) is a generic point of $C(d)$. 
 Since $d a_0 b_{00} e_{ij} \to d a_0 b_{00} e_{ij}^* \to d a_0 b_{00} e_{1j}$, also the point $(a_0^1, b_{00}^1, e_{ij}^*)$ (and similarly, $(a_0^1, b_{00}^1, e_{ij'}^*)$) is generic on $C(d)$.
 
We now consider three different cases:
\begin{enumerate}[(1)]
\item $a_0^1 b_{00}^1 \da_{de_{ij}^*} (a_i^*)^1 (b_{ij}^*)^1$ and $a_0^1 b_{00}^1 \da_{de_{ij'}^*} (a_i^*)^1 (b_{ij'}^*)^1$;
\item $a_0^1 b_{00}^1 \nda_{de_{ij}^*} (a_i^*)^1 (b_{ij}^*)^1$;
\item $a_0^1 b_{00}^1 \nda_{de_{ij'}^*} (a_i^*)^1 (b_{ij'}^*)^1$.
\end{enumerate} 

In case (1), it follows from Lemma \ref{loytyyperhe} and Claim \ref{binterbounded2} that  
 $$e_{ij}^*, e_{ij'}^* \in \textrm{bcl}(c,d,a_0, b_{00}, a_i^*, b_{ij}^*, b_{ij'}^*)=\textrm{bcl}(c,d,a_0, b_{00}, a_i^*, b_{ij}^*).$$
Applying the exchange principle repeatedly and taking into account that $e_{ij}^* \notin \textrm{bcl}(c,d,a_{0}, b_{00}, a_i^*)$ (by Claim \ref{eij}) and $b_{ij}^* \in \textrm{bcl}(c,d,e_{ij}^*,a_i^*)$, we get that $a_i^*,b_{ij}^* \in \textrm{bcl}(c,d,a_0, b_{00}, e_{ij}^*,e_{ij'}^*)$, as wanted. 
 
For case (2), we note that since $U((a_i^*)^1,(b_{ij}^*)^1/de_{ij}^*)=1$, we have $(a_i^*)^1,(b_{ij}^*)^1 \in \textrm{bcl}(d,e_{ij}^*,a_0^1,b_{00}^1)$.
It follows that  $a_i^*,b_{ij}^* \in \textrm{bcl}(c,d,a_0, b_{00}, e_{ij}^*,e_{ij'}^*)$, as wanted.

For case (3), we get $a_i^*,b_{ij'}^* \in \textrm{bcl}(c,d,a_0, b_{00}, e_{ij}^*,e_{ij'}^*)$ using the same argument as in case (2).
By Claim \ref{binterbounded2}, $b_{ij}^*$ and $b_{ij'}^*$ are interbounded over $c$, so the the claim follows.
\end{proof} 

\begin{claim}\label{gee}
For all $i,j$, $g_{ij}^*=g_{00}$.
\end{claim}

\begin{proof}
By claim \ref{a0b00}, the specialization $A ^0\to_p {A^0}^* \to_p {A^0}''$ gives us
$$cdgb_{ij}g_{ij} \to cdgb_{00}g_{ij}^* \to cdgb_{00} g_{00}.$$

Since $dim(b_{ij}g_{ij}/gc)=dim(b_{00}g_{00}/gc)=1$, we must have $dim(b_{00}g_{ij}^*/gc)=1$,
and thus $g_{ij}^* \in \textrm{bcl}(g_{00}, g, c)$.
It follows that the specialization $cdgb_{00}g_{00}g_{ij}^* \to cdgb_{00}g_{00}g_{00}$ is an isomorphism and thus $g_{ij}^*=g_{00}$.
\end{proof}

\begin{claim}\label{rindisc}
The array ${A^0}^*$ is rank-indiscernible of type $2m+n-2$ over $p$.
\end{claim}
  
\begin{proof}
We will apply Lemma \ref{tekn2} to the arrays $A^0$ and ${A^0}^*$.
By Claim \ref{asstekn2}, the required assumptions hold for $A^0$.
From Claims \ref{eij}, \ref{a0b00} and \ref{gee}, it follows that $Lt({A^0}^*_{ij}/p)$ does not depend on $i,j$, that $dim({A^0}^*_{ij}/p)=1$, and that each row ${A^0}^*_{iJ}$ is independent over $p$.

Thus, we only need to show that for any $i,j,j',j''$, the specialization 
$$A_{ij}A _{ij'}A_{ij''} \to_p A_{ij}^*A_{ij'}^*A_{ij''}^*$$
is strongly good, i.e. that the specialization
\begin{eqnarray*}\label{ekspek}
(p,a_i, b_{ij}, b_{ij'}, b_{ij''}, g_{ij}, g_{ij'}, g_{ij''} e_{ij}, e_{ij'}, e_{ij''}) \to (p, a_0, b_{00}, b_{00}, b_{00}, g_{00}, g_{00}, g_{00}, e_{ij}^*, e_{ij'}^*, e_{ij''}^*)
\end{eqnarray*}
is strongly good.
Now, $pe_{ij} e_{ij'}e_{ij''} \to pe_{ij}^* e_{ij'}^*e_{ij''}^*$ and $a_i \to a_0$ are isomorphisms and thus strongly regular. 
If we manage to show that $a_i$ is independent from $pe_{ij} e_{ij'}e_{ij''}$, then the definition of strongly regular specialization will give us that 
\begin{eqnarray}\label{cdetc}
(p, a_i, e_{ij}, e_{ij'}, e_{ij''}) \to (p, a_0,  e_{ij}^*, e_{ij'}^*, e_{ij''}^*) 
\end{eqnarray}
is strongly regular.
 
It follows from Lemma \ref{loytyyperhe} and the choice of $p$ that $dim(e_{ij}, e_{ij'}, e_{ij''}/p)=3$.
Since the type of the array $A$ over $p$ is $2m+n-1$ and $g_{ij}, g_{ij'},g_{ij''},b_{ij}, b_{ij'}, b_{ij''} \in \textrm{bcl}(a_i, e_{ij}, e_{ij'}, e_{ij''},p)$,
we can calculate that $dim(a_i, e_{ij}, e_{ij'}, e_{ij''}/p)=4$, so $a_i$ is independent from $pe_{ij} e_{ij'}e_{ij''}$ and the specialization (\ref{cdetc}) is strongly regular.

Further, we have $b_{ij}, g_{ij} \in \textrm{bcl}(p, a_i, e_{ij})$.
Since $(d, a_0, b_{00}, e_{ij}^*)$ is a generic point of $C$, the specialization $p a_i e_{ij} \to p a_0 e_{ij}^*$ is an isomorphism and thus strongly good. 
The analogues hold for $j,j'$, so we may apply the recursive definition of strongly good specializations (Definition \ref{good}) to show that the specialization  
$$(p, a_i, b_{ij}, b_{ij'}, b_{ij''}, e_{ij}, e_{ij'}, e_{ij''}, g_{ij}, g_{ij'}, g_{ij''}) \to (p, a_0, b_{00}, b_{00}, b_{00}, e_{ij}^*, e_{ij'}^*, e_{ij''}^*, g_{00}, g_{00}, g_{00}))$$
is strongly good, as wanted.
\end{proof}   

We will apply Axiom (9) of Zariski-like structures to the specialization $A \to_p A''$ to obtain an infinite indiscernible array of type $2m+n-2$ over $p$.
Enumerate the elements on the left side of the specialization so that $a_0$ is the element enumerated by $0$ and $a_1$ the element enumerated by $1$, and use a corresponding enumeration on the right side (there, both the element enumerated by $0$ and the element enumerated by $1$ will be $a_0$).
Let $S$ be the collection of index sets corresponding to all $m \times n$ subarrays of $A$ containing the entry $A_{11}$ for all natural numbers $m,n$, and add $0$ to every $X \in S$.
The set $S$ is unbounded and directed.  
 
Condition (i) of Axiom (9) clearly holds for $S$.
Now, every $Y \in S$ corresponds to subarrays $A_Y$ and $A_Y''$ of $A$ and $A''$, respectively (we get the correspondence by removing the element indexed with $0$ from $Y$), and by Claim \ref{asstekn2}, $A_Y$ is rank indiscernible of type $2m+n-1$ over $p$.
Suppose now that $A_Y'$ is an array such that $A_Y \to_p A_Y' \to_p A_Y''$, the specialization $A_Y \to_p A_Y'$ has rank $1$,
and both the element enumerated by $0$ and $1$ in $A_Y'$ are $a_0$.
To prove Claim \ref{rindisc} we only used the fact that the conditions listed above for $A_Y'$ hold for the array $A_Y^*$.
Thus, the same argument can be applied to $A_Y'$ and we get that $A_Y'$ is rank indiscernible of type $2m+n-2$ over $p$.
Now, the specialization $A_Y \to_p A_Y^*$ has rank $1$, and by the rank indiscernibility of the arrays, so does any specialization of subarrays $A_X \to_p A_X^*$ where $X \in S$ and $X \subset Y$.
Hence, also condition (ii) of Axiom (9) holds for $S$.

Now, Axiom (9) gives us a sequence $A^*$ so that $A \to_p A^* \to_p A''$,
and we can read $A^*$ as an array by taking the sets in $S$ to consist of indices for subarrays that contain $A_{11}^*$.
By Axiom (9), all of these subarrays are of rank $2m+n-2$ over $p$.
We claim that $A^*$ is rank indiscernible of type $2m+n-2$ over $p$.
Indeed, let $A_0$ be an arbitrary $m_0  \times n_0$ subarray of $A^*$.
Then, there is some $(m_0+1) \times (n_0+1)$ subarray $A_1$ of $A^*$ such that $A_1$ contains the entry $A_{11}^*$ and $A_0$ is a subarray of $A_1$. 
By the same argument that is used to prove Claim \ref{rindisc}, 
$A_1$ is rank indiscernible of type $2m+n-2$ over $p$.
 
We wish to apply Lemma \ref{groupexists} to show that there is an algebraically closed field in $(\M^{eq})^{eq}$.
For this, we need $A^*$ to be indiscernible.
However, if in the beginning of the proof, when we started constructing the sequences $\alpha_i$, $\beta_i$ and $\beta^j$ for $i,j<\kappa$, we have chosen the cardinals $\kappa$ and $\lambda$ to be large enough, we may assume that $A$ and thus $A^*$ is big enough (rather than the size $\omega_1 \times \omega_1$ which we have used so far for convenience of notation) that we may apply the Shelah trick again
(first done in the second paragraph after the proof of Claim \ref{lisatty}).
Thus, we may without loss suppose that $A^*$ is indiscernible.
By Lemma \ref{groupexists}, there is an algebraically closed field in $(\M^{eq})^{eq}$.
\end{proof}

\end{document}